\documentclass{amsart}

\newtheorem{theorem}[equation]{Theorem}
\newtheorem{lemma}[equation]{Lemma}
\newtheorem{corollary}[equation]{Corollary}
\newtheorem{proposition}[equation]{Proposition}

\numberwithin{equation}{section}

\usepackage{amscd,amsmath,amssymb,verbatim}

\begin{document}

\title[Monomial deformations of generalized Delsarte polynomials]{On monomial deformations of generalized \\ Delsarte polynomials}
\author{Alan Adolphson}
\address{Department of Mathematics\\
Oklahoma State University\\
Stillwater, Oklahoma 74078}
\email{adolphs@math.okstate.edu}
\author{Steven Sperber}
\address{School of Mathematics\\
University of Minnesota\\
Minneapolis, Minnesota 55455}
\email{sperber@math.umn.edu}
\date{\today}
\begin{abstract}
By a generalized Delsarte polynomial we mean a Laurent polynomial whose exponent vectors are linearly independent.
We consider certain monomial deformations of generalized Delsarte polynomials and study their associated differential modules.   
We determine the solutions at the origin, which are all classical ${}_kF_{k-1}$-hypergeometric functions with the variable raised to a power.  Using standard changes of variable, we also obtain the solutions at infinity.
\end{abstract}
\maketitle

\section{Introduction}

Let $A = \{{\bf a}_j\}_{j=1}^m\subseteq{\mathbb Z}^n$ be a linearly independent set in ${\mathbb R}^n$.  We assume there is a point ${\bf a}_0\in{\mathbb Z}^n$ lying in the interior of the convex hull of the set $A$.  This is equivalent to the existence of a relation
\begin{equation}
\ell_0{\bf a}_0 = \sum_{j=1}^m \ell_j{\bf a}_j
\end{equation}
where $\ell_0,\ell_1,\dots,\ell_m$ are positive integers having greatest common divisor equal to one satisfying
\begin{equation}
\ell_0 = \sum_{j=1}^m \ell_j.
\end{equation}
Set $A_+ = A\cup\{{\bf a}_0\}$.  We associate a differential module to the polynomial
\begin{equation}
f_\lambda = \sum_{j=1}^m \ell_jx^{{\bf a}_j} - \ell_0\lambda x^{{\bf a}_0}\in{\mathbb Z}[\lambda][x_1^{\pm 1},\dots, x_n^{\pm 1}],
\end{equation}
where $\lambda$ is an indeterminate.  (In the special case of invertible polynomials, the $\ell_j$, $j=1,\dots,m$, are the dual weights and $\ell_0$ is the dual degree.)

Let $V\subseteq {\mathbb R}^n$ be the subspace generated by $A$, so $\dim_{\mathbb R}V = m$.   Let $V_{\mathbb Z} = V\cap{\mathbb Z}^n$ and let $C(A)\subseteq V$ be the real cone generated by $A$.  Put $M=V_{\mathbb Z}\cap C(A)$ and let $S\subseteq {\mathbb C}[\lambda][x_1^{\pm 1},\dots x_n^{\pm 1}]$ be the ${\mathbb C}[\lambda]$-algebra generated by $\{x^u\mid u\in M\}$.  The differential operators
\[ D_i:= x_i\frac{\partial}{\partial x_i} + x_i\frac{\partial f_\lambda}{\partial x_i}\; \bigg(=\frac{1}{\exp f_\lambda}\circ x_i\frac{\partial}{\partial x_i}\circ \exp f_\lambda\bigg) \]
for $i=1,\dots,n$ operate on $S$ and commute with the operator 
\[ D_\lambda := \frac{\partial}{\partial\lambda} -\ell_0x^{{\bf a}_0}\;\bigg( =\frac{1}{\exp f_\lambda}\circ \frac{\partial}{\partial \lambda}\circ \exp f_\lambda\bigg) , \]
which also operates on $S$.  Let ${\mathcal D}:={\mathbb C}[\lambda]\langle \partial_\lambda\rangle$ be the ring of differential operators with coefficients in ${\mathbb C}[\lambda]$.  Then both $S$ and the ${\mathbb C}[\lambda]$-module ${\mathcal W}:= S/\sum_{i=1}^n D_iS$
are left ${\mathcal D}$-modules, where $\partial_\lambda$ acts on $S$ as $D_\lambda$.  For $\xi\in S$, we also denote by $\xi$ its image in the quotient ${\mathcal W}$.  The context should prevent this abuse of notation from causing confusion.

Let $S'\subseteq {\mathbb C}(\lambda)[x_1^{\pm 1},\dots x_n^{\pm 1}]$ be the ${\mathbb C}(\lambda)$-algebra generated by $\{x^u\mid u\in M\}$ and let ${\mathcal W}' =S'/\sum_{i=1}^n D_iS'$.  Both $S'$ and ${\mathcal W}'$ are left ${\mathcal D}'$-modules, where ${\mathcal D}' = {\mathbb C}(\lambda)\langle\partial_\lambda\rangle$.  For $\xi\in S'$, we also denote by $\xi$ its image in the quotient ${\mathcal W}'$.  

Our main result is an explicit construction of a basis for the finite-dimensional ${\mathbb C}$-vector space ${\rm Hom}_{\mathcal D}({\mathcal W},{\mathbb C}[[\lambda]])$, the formal solutions of~${\mathcal W}$ at the origin, in terms of classical ${}_kF_{k-1}$-hypergeometric series for various $k$ (Corollary~3.18 and Theorem~4.7).  We apply this result to two related questions.  First of all, the quotient ${\mathcal W}'$ is a finite-dimensional ${\mathbb C}(\lambda)$-vector space, so for any basis there is an associated Picard-Fuchs equation.  We find the fundamental solution matrix for one such basis, again in terms of classical ${}_kF_{k-1}$-hypergeometric series.  As a second application, we compute the hypergeometric equations satisfied by elements $x^{\bf b}$ of this basis for~${\mathcal W}'$.  Using standard changes of variables, we also explain how to find the solutions of ${\mathcal W}'$ at infinity, which may involve $\log \lambda$.  

Special cases of ${\mathcal W}$ and ${\mathcal W}'$ have been studied by various authors.  The relation between these spaces and de Rham cohomology has been known since Dwork \cite{D} and Katz \cite{K} (see also Batyrev \cite{B} and 
Adolphson-Sperber~\cite{AS}).
The problem of expressing solutions of the Picard-Fuchs equation for certain hypersurfaces in terms of ${}_kF_{k-1}$-hypergeometric functions or for determining the differential operator satisfied by a differential form has also been studied by Kloosterman \cite{Kl}, G\"ahrs~\cite{G}, Miyatani \cite{M}, Doran et al.\ \cite{Do}, and Negishi \cite{N}.  We give an example of the application of our results to such a computation in Section~11.

In the case of certain monomial deformations of diagonal hypersurfaces, Katz~\cite{K2} makes a normalization of the equation $f_\lambda = 0$ that puts the singular points at roots of unity and leads to an expression of cohomology in terms of hypergeometric sheaves.  Our normalization~(1.1), a special case of \cite[Equation~(1.1)]{AS2}, generalizes Katz's normalization.  This observation motivated us to combine Dwork's original method \cite{D} with some ideas from \cite{AS2} and \cite[Section~3]{AS1} to calculate the solutions.

We describe the type of hypergeometric equations that arise here.  Let $\delta_x = x\frac{d}{dx}$ and consider the differential operator
\begin{equation}
 (\delta_x+\beta_1-1)\cdots(\delta_x + \beta_k-1) - x(\delta_x+\alpha_1)\cdots(\delta_x+\alpha_k).
\end{equation}
Suppose that the $\beta_j$ are distinct elements in $(0,1]$ with $\beta_1 = 1$.  One solution of the operator (1.4) at the origin is then
\begin{equation}
 {}_kF_{k-1}\bigg(\begin{matrix} \alpha_1,\dots,\alpha_k \\ \beta_2,\dots,\beta_k\end{matrix};x\bigg) := \sum_{s=0}^\infty \frac{(\alpha_1)_s\cdots(\alpha_k)_s}{(\beta_2)_s\cdots(\beta_k)_s(1)_s}x^s. 
 \end{equation}
The other $k-1$ solutions are obtained by multiplying the series on the right-hand side by $x^{1-\beta_j}$, $j=2,\dots,k$, and adding $1-\beta_j$ to the quantity inside each Pochhammer symbol:
\begin{equation}
x^{1-\beta_j}{}_kF_{k-1}\bigg(\begin{matrix} \alpha_1+1-\beta_j,\dots,\alpha_k+1-\beta_j \\ \beta_1+1-\beta_j,.\hat{.}.,\beta_k+1-\beta_j\end{matrix};x\bigg),
\end{equation}
where the symbol \^{} indicates that the term $\beta_j+1-\beta_j=1$ has been omitted (as is usual in hypergeometric notation).  In particular, one can read off the differential operator (1.4) and the remaining solutions (1.6) from the single solution (1.5).  The hypergeometric equations we encounter will be of this type with the variable $x$ replaced by $\lambda^{\ell_0}$.  

We illustrate how the series (1.5) appear in our work.
Let $P(A)\subseteq V$ be the parallelopiped defined by $A$:
\[ P(A) = \bigg\{ \sum_{j=1}^m c_j{\bf a}_j\mid \text{$0\leq c_j<1$ for $j=1,\dots,m$}\bigg\}. \]
Let ${\mathbb Z}A\subseteq V_{\mathbb Z}$ be the abelian group generated by $A$.  The set
${\mathcal B}:= V_{\mathbb Z}\cap P(A)$ is a complete set of coset reprsentatives for ${\mathbb Z}A$ in $V_{\mathbb Z}$.
Put $d=\lvert {\mathcal B}\rvert = [V_{\mathbb Z}:{\mathbb Z}A]$.  
Let ${\mathbb Z}A_+\subseteq V_{\mathbb Z}$ be the subgroup generated by $A_+$.  Equation (1.2) implies that $[{\mathbb Z}A_+:{\mathbb Z}A] = \ell_0$.  

We show (Proposition 5.34) that the set $X^{\mathcal B}:= \{x^{\bf b}\mid {\bf b}\in{\mathcal B}\}$ is a basis for~${\mathcal W}'$.  In Section~6, we construct a fundamental solution matrix for the Picard-Fuchs equation associated to this basis.  Its rows and columns are indexed by the set ${\mathcal B}$, and series of the type (1.5) with $x$ replaced by $\lambda^{\ell_0}$ appear as the diagonal entries. Constant multiples of the series (1.6) with $x$ replaced by $\lambda^{\ell_0}$ are the remaining nonzero entries in that row.

These diagonal entries are easy to describe.  Let ${\bf b}\in{\mathcal B}$ and write ${\bf b} = \sum_{j=1}^m v_j{\bf a}_j$ with $v_j\in[0,1)\cap{\mathbb Q}$ for $j=1,\dots,m$.  The $({\bf b},{\bf b})$-entry in our solution matrix is
\begin{equation}
\sum_{s=0}^\infty \frac{\displaystyle \prod_{j=1}^m \bigg(\frac{v_j}{\ell_j}\bigg)_s\bigg(\frac{v_j+1}{\ell_j}\bigg)_s\cdots\bigg(\frac{v_j+\ell_j-1}{\ell_j}\bigg)_s}{\displaystyle \bigg(\frac{1}{\ell_0}\bigg)_s\bigg(\frac{2}{\ell_0}\bigg)_s\cdots\bigg(\frac{\ell_0}{\ell_0}\bigg)_s}\lambda^{s\ell_0}. 
\end{equation}

In general there is cancellation in this quotient.  Let ${\mathcal B}_1\subseteq{\mathcal B}$ be the subset consisting of those points that are congruent to ${\bf b}$ modulo ${\mathbb Z}A_+$, a set of cardinality~$\ell_0$, and let $R$ be the number of elements of ${\mathcal B}_1$ that are interior to $P(A)$.  If ${\bf b}$ is an interior point of $P(A)$, then (1.7) will be an ${}_RF_{R-1}$-hypergeometric function.  If ${\bf b}$ is a boundary point of $P(A)$, then (1.7) will be an ${}_{R+1}F_R$-hypergeometric function. (If ${\bf b}$ is a boundary point, then $v_j=0$ for some $j$, so (1.7) equals the constant function~1.)  This follows from the fact (Proposition~6.14) that in either case, if ${\bf b}'\in{\mathcal B}_1$ is a boundary point of $P(A)$ and ${\bf b} + t {\bf a}_0\equiv {\bf b}' \pmod{{\mathbb Z}A}$ with $t\in\{1,2,\dots,\ell_0-1\}$, then the factor $(1-\frac{t}{\ell_0})_s$ cancels in (1.7).  Using (1.6), one can then read off the other solutions of the corresponding Equation~(1.4).  Constant multiples of these solutions with the variable $x$ replaced by $\lambda^{\ell_0}$ are the nonzero entries in the row corresponding to ${\bf b}$.  

Proposition~6.14 describes the cancellation algorithm.  Theorems~6.17 and~6.18 describe the entries in row ${\bf b}$.  In Section~8 we show that the element $x^{\bf b}\in{\mathcal W}'$ satisfies the differential equation whose solutions are the nonzero entries in row~${\bf b}$.

{\bf Example.}  Take $f_\lambda = 3x_1^2 + 2x_2^3 +x_3^6 - 6\lambda x_1x_2x_3$.   We have ${\bf a}_1 = (2,0,0)$, ${\bf a}_2 = (0,3,0)$, ${\bf a}_3 = (0,0,6)$, and ${\bf a}_0 = (1,1,1)$, giving the relation
\[ 6{\bf a}_0 = 3{\bf a}_1 + 2{\bf a}_2 + {\bf a}_3. \]
Thus $\ell_1 = 3$, $\ell_2 = 2$, $\ell_3 = 1$, $\ell_0 = 6$, and $\lvert{\mathcal B}\rvert = 36$.  Let ${\mathcal B}_1$ be the subset of ${\mathcal B}$ consisting of elements congruent to $(0,0,0)$ modulo ${\mathbb Z}A_+$:
\[ {\mathcal B}_1 = \{ (0,0,0), (1,1,1), (0,2,2), (1,0,3), (0,1,4), (1,2,5) \}. \]
In this case $R=2$, so the row corresponding to ${\bf b} = (0,0,0)$ has three nonzero entries.  We have $v_j=0$ for $j=1,2,3$, so (1.7) becomes
\[ \sum_{s=0}^\infty \frac{(\frac{0}{3})_s(\frac{1}{3})_s(\frac{2}{3})_s(\frac{0}{2})_s(\frac{1}{2})_s(\frac{0}{6})_s}{(\frac{1}{6})_s(\frac{2}{6})_s(\frac{3}{6})_s(\frac{4}{6})_s(\frac{5}{6})_s(\frac{6}{6})_s}\lambda^{6s} =
{}_3F_2(0,0,0;1/6,5/6;\lambda^6) = 1. \]
The cancellation in this quotient is determined by the boundary points (other than $(0,0,0)$) in ${\mathcal B}_1$.  We have
\[ (0,0,0) + \begin{cases} 2(1,1,1) \\ 3(1,1,1) \\ 4(1,1,1) \end{cases} \equiv \begin{cases} (0,2,2) \\ (1,0,3) \\ (0,1,4) \end{cases} \pmod{{\mathbb Z}A}. \]
By Proposition 6.14(b), the terms that cancel are $(1-\frac{t}{6})_s$ for $t=2,3,4$. 

The other two nonzero entries in this row are given, up to a constant multiple, by (1.6) with $x$ replaced by $\lambda^6$:
\[ \lambda\, {}_3F_2(1/6,1/6,1/6;1/3,7/6;\lambda^6)\quad\text{and}\quad \lambda^5{}_3F_2(5/6,5/6,5/6;5/3,11/6;\lambda^6).  \]
These entries appear in columns $(1,1,1)$ and $(1,2,5)$, respectively.  

Now consider the row corresponding to ${\bf b} = (1,1,1)$.  
We have $v_1 = 1/2$, $v_2 = 1/3$, and $v_3 = 1/6$, so by (1.7) the diagonal entry is
\[ \sum_{s=0}^\infty \frac{(\frac{1}{6})_s(\frac{3}{6})_s(\frac{5}{6})_s(\frac{1}{6})_s(\frac{4}{6})_s(\frac{1}{6})_s}{(\frac{1}{6})_s(\frac{2}{6})_s(\frac{3}{6})_s(\frac{4}{6})_s(\frac{5}{6})_s(\frac{6}{6})_s}\lambda^{6s} =
{}_2F_1(1/6,1/6;1/3;\lambda^6). \]
For the boundary points of ${\mathcal B}_1$ we have
\[ (1,1,1) + \begin{cases} 1(1,1,1) \\ 2(1,1,1) \\ 3(1,1,1) \\ 5(1,1,1) \end{cases} \equiv \begin{cases} (0,2,2) \\ (1,0,3) \\ (0,1,4) \\ (0,0,0) \end{cases}\pmod{{\mathbb Z}A}. \]
By Proposition 6.14(a), the terms that cancel are $(1-\frac{t}{6})_s$ for $t=1,2,3,5$.

From (1.6), the entry in the column corresponding to $(1,2,5)$ is, up to a constant multiple, $\lambda^4 {}_2F_1(5/6,5/6;5/3;\lambda^6)$.  

\section{Solutions at the origin}

Let $S_0\subseteq{\mathbb C}[x_1^{\pm 1},\dots,x_n^{\pm 1}]$ be the ${\mathbb C}$-algebra generated by $\{x^u\mid u\in M\}$.  Let $f_0$ be the Laurent polynomial obtained from $f_\lambda$ by setting $\lambda=0$:
\[ f_0 = \sum_{j=1}^m \ell_j x^{{\bf a}_j}\in S_0. \]
Put for $i=1,\dots,n$
\[ D_{i,0} = x_i\frac{\partial}{\partial x_i} + x_i\frac{\partial f_0}{\partial x_i}, \]
operators on $S_0$.  
Finally, let ${\mathcal W}_0 = S_0/\sum_{i=1}^n D_{i,0}S_0$.  
\begin{theorem}
There is an isomorphism of ${\mathbb C}$-vector spaces
\begin{equation}
{\rm Hom}_{\mathcal D}({\mathcal W},{\mathbb C}[[\lambda]])\cong {\rm Hom}_{\mathbb C}({\mathcal W}_0,{\mathbb C}).
\end{equation}
\end{theorem}

The isomorphism (2.2) is a special case of well-known, more general results (see, for example, \cite[Section 4]{A1}).  We follow here an idea of Dwork \cite[Section~5(c)]{D} to make this isomorphism explicit, so that by computing a basis of the simpler right-hand side we obtain explicit formulas for a basis of the left-hand side.  We begin with some general observations. 

Let ${\mathcal F}$ be a left ${\mathcal D}$-module and consider the ${\mathbb C}[\lambda]$-module
\[ R({\mathcal F}) = \bigg\{\sum_{u\in M} A_ux^{-u}\mid A_u\in{\mathcal F}\bigg\}. \]
Consider the pairing $R({\mathcal F})\times S\to{\mathcal F}$ defined by
\[ \bigg\langle \sum_{u\in M} A_ux^{-u},\sum_{u\in M} B_ux^u\bigg\rangle = \sum_{u\in M} B_uA_u, \]
where the $B_u$ lie in ${\mathbb C}[\lambda]$ and the sum on the right-hand side is finite because the second sum on the left-hand side is.  This pairing defines an isomorphism
\begin{equation}
{\rm Hom}_{{\mathbb C}[\lambda]}(S,{\mathcal F})\cong R({\mathcal F}).
\end{equation}

We next determine which elements of $R({\mathcal F})$ correspond to elements of ${\rm Hom}_{\mathcal D}(S,{\mathcal F})$.  
The condition to be satisfied is that for all $u\in M$ we have
\begin{align*}
 \bigg\langle \sum_{v\in M} A_vx^{-v},D_\lambda(x^u)\bigg\rangle &= \partial_\lambda\bigg\langle \sum_{v\in M} A_vx^{-v},x^u\bigg\rangle\\
  &= \partial_\lambda(A_u).
  \end{align*}
But $D_\lambda(x^u) = -\ell_0x^{u+{\bf a}_0}$, so the left-hand side is just $-\ell_0A_{u+{\bf a}_0}$.  Thus if we put
\[ R^*({\mathcal F}) = \bigg\{ \sum_{u\in M} A_ux^{-u}\mid \text{$A_u\in{\mathcal F}$ and $\partial_\lambda(A_u) = -\ell_0A_{u+{\bf a}_0}$ for all $u\in M$}\bigg\}, \]
then
\begin{equation}
{\rm Hom}_{\mathcal D}(S,{\mathcal F})\cong R^*({\mathcal F}).
\end{equation}

We want to compute ${\rm Hom}_{\mathcal D}({\mathcal W},{\mathcal F})$, which is the intersection of ${\rm Hom}_{\mathcal D}(S,{\mathcal F})$ and ${\rm Hom}_{{\mathbb C}[\lambda]}({\mathcal W},{\mathcal F})(\subseteq {\rm Hom}_{{\mathbb C}[\lambda]}(S,{\mathcal F}))$.  We compute ${\rm Hom}_{{\mathbb C}[\lambda]}({\mathcal W},{\mathcal F})$.  By (2.3), these are the elements of $R({\mathcal F})$ that vanish on $\sum_{i=1}^n D_iS$.  Let
\[ \xi = \sum_{u\in M} A_ux^{-u}\in R({\mathcal F}). \]
Then $\xi$ vanishes on $\sum_{i=1}^n D_iS$ if and only if it vanishes on $D_i(x^v)$ for all $i=1,\dots,n$ and all $v\in M$.  
  So the condition to be satisfied is (where ${\bf a}_j = (a_{1j},\dots,a_{nj})$)
 \begin{align}
 \langle \xi,D_i(x^v)\rangle &= \bigg\langle \xi,v_ix^v -\ell_0a_{i0}\lambda x^{v+{\bf a}_0} + \sum_{j=1}^m \ell_ja_{ij}x^{v+{\bf a}_j}\bigg\rangle \\
  &= v_iA_v -\ell_0a_{i0}\lambda A_{v+{\bf a}_0} + \sum_{j=1}^m \ell_ja_{ij}A_{v+{\bf a}_j} \nonumber \\
   &= 0.\nonumber
   \end{align}
   
   Let $\gamma_-$ be defined by ${\mathbb C}[\lambda]$-linearity and the condition
\[ \gamma_-(x^{-u}) = \begin{cases}  x^{-u} & \text{if $u\in M$,} \\ 0 & \text{if $u\not\in M$.} \end{cases} \]
Let $D_i^*$ be the operator on $R({\mathcal F})$ defined for $i=1,\dots,n$ by
\[ D_i^* = \gamma_-\circ \bigg(-x_i\frac{\partial}{\partial x_i} - \ell_0a_{i0}\lambda x^{{\bf a}_0} +\sum_{j=1}^m \ell_ja_{ij}x^{{\bf a}_j}\bigg). \]
 For $v\in M$ the coefficient of $x^{-v}$ in $D_i^*(\xi)$ is
 \[ v_iA_v -\ell_0a_{i0}\lambda A_{v+{\bf a}_0} + \sum_{j=1}^m \ell_ja_{ij}A_{v+{\bf a}_j}. \]
 It now follows from (2.5) that 
 \[ {\rm Hom}_{{\mathbb C}[\lambda]}({\mathcal W},{\mathcal F}) = \{ \xi \in R({\mathcal F})\mid \text{$D_i^*(\xi) = 0$ for $i=1,\dots,n$}\}. \]
   
 Define
 \begin{multline*} {\mathcal K}({\mathcal F}) = \bigg\{ \xi = \sum_{u\in M} A_ux^{-u}\mid \text{$A_u\in{\mathcal F}$ for $u\in M$, $D_i^*(\xi) = 0$ for $i=1,\dots,n$,} \\
 \text{ and $\partial_\lambda(A_u) =-\ell_0A_{u+{\bf a}_0}$ for all $u\in M$}\bigg\}. 
 \end{multline*}
We have proved the following result.
\begin{proposition}
${\mathcal K}({\mathcal F})\cong {\rm Hom}_{\mathcal D}({\mathcal W},{\mathcal F})$.
\end{proposition}

We analyze similarly ${\rm Hom}_{\mathbb C}({\mathcal W}_0,Y)$ for an arbitrary ${\mathbb C}$-vector space $Y$.  Let
\[ R_0(Y) = \bigg\{ \sum_{u\in M} A_ux^{-u}\mid A_u\in Y\bigg\}. \]
The pairing $R_0(Y)\times S_0\to Y$ defined by 
 \[ \bigg\langle \sum_{u\in M} A_ux^{-u},\sum_{u\in M} B_ux^u\bigg\rangle = \sum_{u\in M} B_uA_u, \]
where the $B_u$ lie in ${\mathbb C}$, defines an isomorphism
\begin{equation}
{\rm Hom}_{\mathbb C}(S_0,Y)\cong R_0(Y).
\end{equation}

Let $\xi = \sum_{u\in M} A_ux^{-u}\in R_0(Y)$.  Then $\xi$ vanishes on $\sum_{i=1}^n D_{i,0}S_0$ if and only if it vanishes on $D_{i,0}(x^v)$ for all $i=1,\dots,n$ and all $v\in M$.  The condition to be satisfied is
\begin{align} 
\langle \xi,D_{i,0}(x^v)\rangle &= \bigg\langle \xi,v_ix^v + \sum_{j=1}^m \ell_ja_{ij}x^{v+{\bf a}_j}\bigg\rangle\\
 &= v_i A_v + \sum_{j=1}^m \ell_ja_{ij}A_{v+{\bf a}_j} \nonumber \\
  &= 0. \nonumber
  \end{align}

Let $D_{i,0}^*$ be the operator on $R_0(Y)$ defined for $i=1,\dots,n$ by
\[ D_{i,0}^* = \gamma_-\circ\bigg(-x_i\frac{\partial}{\partial x_i} + \sum_{j=1}^m \ell_ja_{ij}x^{{\bf a}_j}\bigg). \]
For $v\in M$, the coefficient of $x^{-v}$ in $D_{i,0}^*(\xi)$ is
\[ v_iA_v + \sum_{j=1}^m \ell_ja_{ij}A_{v+{\bf a}_j}. \]
If we put 
\[ {\mathcal K}_0(Y) = \bigg\{ \xi = \sum_{u\in M} A_ux^{-u}\mid \text{$A_u\in Y$ for $u\in M$ and $D_{i,0}^*(\xi)=0$ for $i=1,\dots,n$}\bigg\}, \]
then from Equation (2.8) we get the following result.
\begin{proposition}
${\mathcal K}_0(Y)\cong {\rm Hom}_{\mathbb C}({\mathcal W}_0,Y)$.
\end{proposition}

\begin{proof}[Proof of Theorem $2.1$]
By Propositions 2.6 and 2.9 we need to prove that
\begin{equation}
{\mathcal K}({\mathbb C}[[\lambda]])\cong {\mathcal K}_0({\mathbb C}).
\end{equation}
For $\xi = \sum_{u\in M} c_ux^{-u}\in{\mathcal K}_0({\mathbb C})$, define
\begin{equation}
 \phi(\xi) = \gamma_-(\exp(-\ell_0\lambda x^{{\bf a}_0})\xi) =:\sum_{u\in M} A_u(\lambda)x^{-u}, 
 \end{equation}
where the $A_u(\lambda)$ lie in ${\mathbb C}[[\lambda]]$.  
We claim that $\phi$ is an isomorphism from ${\mathcal K}_0({\mathbb C})$ onto ${\mathcal K}({\mathbb C}[[\lambda]])$.  

We first show that $\phi(\xi)\in{\mathcal K}({\mathbb C}[[\lambda]])$.  We have
\[ \partial_\lambda(\phi(\xi)) = \sum_{u\in M}(\partial_\lambda A_u)(\lambda) x^{-u}. \]
Since $\partial_\lambda$ commutes with $\gamma_-$ we also have
\begin{align*}
\partial_\lambda(\phi(\xi)) &= \gamma_-(-\ell_0x^{{\bf a}_0}\exp(-\ell_0\lambda x^{{\bf a}_0})\xi)\\
 &= \gamma_-(-\ell_0x^{{\bf a}_0}\phi(\xi))\\
  &= -\ell_0\sum_{u\in M} A_{u+{\bf a}_0}(\lambda)x^{-u}.
 \end{align*}
 It follows that $(\partial_\lambda A_u)(\lambda) = -\ell_0A_{u+{\bf a}_0}(\lambda)$ for all $u\in M$.  
 
 We also need to check that $D_i^*(\phi(\xi)) = 0$ for $i=1,\dots,n$.  First note that
 \[ x_i\frac{\partial}{\partial x_i}\bigg(\exp(-\ell_0\lambda x^{{\bf a}_0})\xi\bigg) = \exp(-\ell_0\lambda x^{{\bf a}_0})\bigg(x_i\frac{\partial\xi}{\partial x_i}-\ell_0\lambda a_{i0}x^{{\bf a}_0}\bigg). \]
We thus have
\begin{multline*}
 D_i^*(\phi(\xi)) = \\ 
 \gamma_-\bigg(\exp(-\ell_0\lambda x^{{\bf a}_0}) \bigg(\bigg(-x_i\frac{\partial\xi}{\partial x_i}+\ell_0\lambda a_{i0}x^{{\bf a}_0}\bigg)-\ell_0a_{i0}\lambda x^{{\bf a}_0}\xi + \sum_{j=1}^m \ell_ja_{ij}x^{{\bf a}_j}\xi\bigg)\bigg). 
 \end{multline*}
Hence
\begin{align*}
 D_i^*(\phi(\xi)) &=\gamma_-\bigg(\exp(-\ell_0\lambda x^{{\bf a}_0})\bigg(-x_i\frac{\partial\xi}{\partial x_i}+ \sum_{j=1}^m \ell_ja_{ij}x^{{\bf a}_j}\xi\bigg)\bigg)\\
  &= \gamma_-\bigg(\exp(-\ell_0\lambda x^{{\bf a}_0})D_{i,0}^*(\xi)\bigg) \\
   &= 0
   \end{align*}
since $\xi\in{\mathcal K}_0({\mathbb C})$.

For $\sum_{u\in M} A_u(\lambda)x^{-u}\in{\mathcal K}({\mathbb C}[[\lambda]])$ define
\[ \phi'\bigg(\sum_{u\in M} A_u(\lambda)x^{-u}\bigg) = \sum_{u\in M} A_u(0)x^{-u}, \]
i.~e., we evaluate all the $A_u(\lambda)$ at $\lambda = 0$.  We claim that $\phi'$ is the inverse of $\phi$.  It is straightforward to check that 
\[ D^*_{i,0}\bigg(\sum_{u\in M} A_u(0)x^{-u}\bigg) = 0 \]
 for $i=1,\dots,n$, so $\phi'$ maps ${\mathcal K}({\mathbb C}[[\lambda]])$ into ${\mathcal K}_0({\mathbb C})$.  It follows from the definition of $\phi$ that if $\xi=\sum_{u\in M} c_ux^{-u}\in{\mathcal K}_0({\mathbb C})$, then $\phi'(\phi(\xi)) = \xi$, so $\phi'$ is onto.  The map $\phi'$ is also injective because $\sum_{u\in M} A_u(\lambda)x^{-u}\in{\mathcal K}({\mathbb C}[[\lambda]])$ is determined by its constant terms $\sum_{u\in M} A_u(0)x^{-u}$.  To see this, suppose that $c_k$ is the coefficient of $\lambda^k$ in the series $A_u(\lambda)$.  The condition $(\partial_\lambda A_u)(\lambda) = -\ell_0 A_{u+{\bf a}_0}(\lambda)$ for all $u\in M$ implies that the constant term of $(-\ell_0)^k A_{u + k{\bf a}_0}(\lambda)$ equals $ k! c_k$.  So if the constant terms $A_u(0)$ vanish for all $u\in M$ then the series $A_u(\lambda)$ vanish for all $u\in M$.
\end{proof}

\section{Basis for ${\rm Hom}_{\mathcal D}({\mathcal W},{\mathbb C}[[\lambda]])$}

In this section we find a basis for ${\mathcal W}_0$, compute the dual basis of ${\rm Hom}_{\mathbb C}({\mathcal W}_0,{\mathbb C})$, and use the isomorphism $\phi$ of the previous section to get a basis for ${\rm Hom}_{\mathcal D}({\mathcal W},{\mathbb C}[[\lambda]])$.  Let ${\mathcal B}$ be as in Section 1 and let $X^{\mathcal B}=\{ x^{\bf b}\mid {\bf b}\in{\mathcal B}\}$.  

Let $\boldsymbol{\alpha} = (a_{ij})$ ($i=1,\dots,n$, $j=1,\dots,m$) be the $n\times m$ matrix whose columns are the ${\bf a}_j$.  Since the ${\bf a}_j$ are linearly independent, there exists an $n\times n$ matrix $\boldsymbol{\beta} = (\beta_{ij})$ of rational numbers such that
\begin{equation}
\boldsymbol{\beta\alpha} = \begin{pmatrix} 1 & 0 & \dots & 0  & 0\\ 0 & 1 & \dots & 0 & 0 \\ \vdots & \vdots & \cdots & \vdots & \vdots \\ 0 & 0 & \dots & 0 & 1\\ 0 & 0 & \dots & 0 & 0 \\ \vdots & \vdots & \vdots & \vdots & \vdots \\ 0 & 0 & \dots & 0 & 0\end{pmatrix}. 
\end{equation}

We define differential operators $\tilde{D}_{i,0}$, $i=1,\dots,n$, on $S_0$ by the formula
\begin{equation}
\begin{pmatrix} \tilde{D}_{1,0} \\ \vdots \\ \tilde{D}_{n,0} \end{pmatrix} = \boldsymbol{\beta}\begin{pmatrix} {D}_{1,0} \\ \vdots \\ {D}_{n,0} \end{pmatrix} = \begin{pmatrix} \sum_{j=1}^n \beta_{1j}D_{j,0} \\ \vdots \\ \sum_{j=1}^n \beta_{nj}D_{j,0} \end{pmatrix}.
\end{equation}

\begin{lemma}
Let $u\in M$ and write $u =(u_1,\dots,u_n) = \sum_{j=1}^m u'_j{\bf a}_j$ with $u'_j\in{\mathbb Q}$ for $j=1,\dots,m$.  We have
\[ \tilde{D}_{j,0}(x^u) = \begin{cases} u'_jx^u + \ell_j x^{u+{\bf a}_j} & \text{for $j=1,\dots,m$,} \\
0 & \text{for $j=m+1,\dots,n$.} \end{cases} \]
\end{lemma}

\begin{proof}
We note that
\begin{equation}
 \begin{pmatrix} x_1\frac{\partial f_0}{\partial x_1} \\ \vdots \\ x_n\frac{\partial f_0}{\partial x_n} \end{pmatrix}  = 
\boldsymbol{\alpha}\begin{pmatrix} \ell_1 x^{{\bf a}_1} \\ \vdots \\ \ell_mx^{{\bf a}_m} \end{pmatrix}
\end{equation}
and 
\begin{equation}
\boldsymbol{\alpha}\begin{pmatrix} u'_1 \\ \vdots \\ u'_m \end{pmatrix} = \begin{pmatrix} u_1 \\ \vdots \\ u_n \end{pmatrix}.
\end{equation}
From the definition we have
\begin{align} 
\begin{pmatrix} \tilde{D}_{1,0} \\ \vdots \\ \tilde{D}_{n,0} \end{pmatrix}(x^u)  & = \boldsymbol{\beta} \begin{pmatrix} x_1\frac{\partial}{\partial x_1} + x_1 \frac{\partial f_0}{\partial x_1} \\ \vdots \\ x_n\frac{\partial}{\partial x_n} + x_n \frac{\partial f_0}{\partial x_n}\end{pmatrix} (x^u) \\
 &= \boldsymbol{\beta}\begin{pmatrix} u_1 \\ \vdots \\ u_n \end{pmatrix} x^u + \boldsymbol{\beta} \begin{pmatrix}  x_1 \frac{\partial f_0}{\partial x_1} \\ \vdots \\ x_n \frac{\partial f_0}{\partial x_n}\end{pmatrix}x^u \nonumber \\
  &= \boldsymbol{\beta\alpha}\begin{pmatrix} u'_1 \\ \vdots \\ u'_m\end{pmatrix} x^u + \boldsymbol{\beta\alpha}\begin{pmatrix} \ell_1x^{{\bf a}_1} \\ \vdots \\ \ell_mx^{{\bf a}_m} \end{pmatrix} x^u  \nonumber \\
  & = \begin{pmatrix} u'_1x^u + \ell_1 x^{u+{\bf a}_1} \\ \vdots \\ u'_mx^u + \ell_m x^{u+{\bf a}_m} \\ 0 \\ \vdots \\ 0 \end{pmatrix} \nonumber,
 \end{align}
where the third equality follows from (3.4) and (3.5) and the fourth equality follows from (3.1).
\end{proof}

Since $\boldsymbol{\beta}$ is invertible, Equation (3.2) and Lemma 3.3 imply the following result.
\begin{corollary}
We have ${\mathcal W}_0 = S_0/\sum_{j=1}^m \tilde{D}_{j,0}S_0$.
\end{corollary}

The following result is clear.
\begin{lemma}
 The set $X^{\mathcal B}$ is a basis for $S_0/\sum_{j=1}^m x^{{\bf a}_j}S_0$.
 \end{lemma}
 
The set $A_+$ lies on a hyperplane in $V$, which allows us to define a grading on the spaces $S$, $S'$, and $S_0$.  For $u = (u_1,\dots,u_n)\in M$ we write $u=\sum_{j=1}^m u'_j{\bf a}_j$ and define $\deg x^u = \sum_{j=1}^m u'_j$.  Equivalently, if $\sum_{i=1}^n h_iy_i=1$ is the equation of a hyperplane containing~$A_+$, then $\deg x^u = \sum_{i=1}^n h_iu_i$.  The degree is a surjective map from $\{x^u\mid u\in M\}$ onto the set $(q_0)^{-1}{\mathbb Z}_{\geq 0}$ for some positive integer $q_0$.  In the study of invertible polynomials, the integer $q_0$ is the degree of the polynomial and the $q_0h_i$, $i=1,\dots,n$, are the weights.  

{\bf Notation:}  Let $L$ be a module (over some ring) and let $\{\phi_j\}_{j=1}^m$ be commuting endomorphisms of $L$.  We denote by ${\rm Kos}^\bullet(L,\{\phi\}_{j=1}^m)$ the associated (cohomological) Koszul complex.  More precisely, let the $e_i$, $i=1,\dots,m$ be formal symbols and define
\[ {\rm Kos}^k(L,\{\phi\}_{j=1}^m) = \bigoplus_{1\leq j_1<\cdots<j_k\leq m} L\,e_{j_1}\wedge\cdots\wedge e_{j_k}. \]
The boundary map $\nabla:{\rm Kos}^k(L,\{\phi\}_{j=1}^m)\to {\rm Kos}^{k+1}(L,\{\phi\}_{j=1}^m)$ is given for $\xi\in L$ by
\[ \nabla(\xi\,e_{j_1}\wedge\cdots\wedge e_{j_k}) = \bigg(\sum_{j=1}^m \phi_j(\xi)\,e_j\bigg)\wedge e_{j_1}\wedge\cdots\wedge e_{j_k}. \]

\begin{proposition}
The set $X^{\mathcal B}$ is a basis for ${\mathcal W}_0$.
\end{proposition}

\begin{proof}
Consider the Koszul complex ${\rm Kos}^\bullet(S_0,\{\tilde{D}_{j,0}\}_{j=1}^m)$ defined by the $\tilde{D}_{j,0}$.  The above grading makes it into a filtered complex for which 
\[ H^m({\rm Kos}^\bullet(S_0,\{\tilde{D}_{j,0}\}_{j=1}^m)) = {\mathcal W}_0. \]
By Lemma 3.3 the associated graded complex is ${\rm Kos}^\bullet(S_0,\{\ell_jx^{{\bf a}_j}\}_{j=1}^m)$.  Since the set $\{x^{{\bf a}_j}\}_{j=1}^m$ is a regular sequence on $S_0$, we have $H^i({\rm Kos}^\bullet(S_0,\{\ell_jx^{{\bf a}_j}\}_{j=1}^m)) = 0$ for $0\leq i<m$.  It follows that ${\mathcal W}_0\cong H^m({\rm Kos}^\bullet(S_0,\{\ell_jx^{{\bf a}_j}\}_{j=1}^m))$.  But
\[ H^m({\rm Kos}^\bullet(S_0,\{\ell_jx^{{\bf a}_j}\}_{j=1}^m)) = S_0/\sum_{j=1}^m x^{{\bf a}_j}S_0, \]
 so Lemma 3.8 implies the proposition.
\end{proof}

We use the following notation, which we find simpler for some formulas than the classical Pochhammer notation.  Later we switch to Pochhammer notation to make the connection with hypergeometric series.  Define for $z\in{\mathbb C}$
\[ [z]_l = \begin{cases} 1 & \text{if $l=0$,} \\ \displaystyle \frac{1}{(z+1)(z+2)\cdots (z+l)} &  \text{if $l>0$ and $z\not\in\{-1,-2,\dots,-l\}$,} \\
z(z-1)\cdots(z+l+1) & \text{if $l<0$.} \end{cases} \]
The relation to the Pochhammer symbol $(z)_l$ is that $[z]_l$ is defined if and only if $(-z)_{-l}$ is defined, in which case
\begin{equation}
[z]_l = (-1)^l(-z)_{-l}.
\end{equation}

We construct the dual basis for ${\rm Hom}_{\mathbb C}({\mathcal W}_0,{\mathbb C})$.  Fix ${\bf b}\in {\mathcal B}$.  Then
\begin{equation}
{\bf b} = \sum_{j=1}^m v_j{\bf a}_j
\end{equation}
 for unique $v_1,\dots,v_m\in[0,1)$.   For $j=1,\dots,m$ we define series
\begin{equation}
 g_j(t) = \sum_{s\in{\mathbb Z}_{\leq 0}} [-v_j]_s t^{-v_j+s}. 
 \end{equation}
  Set
 \begin{equation}
  g_{\bf b} = \prod_{j=1}^m \ell_j^{v_j}g_j(\ell_jx^{{\bf a}_j}).
  \end{equation}
A straightforward calculation establishes the following result.
\begin{lemma}
We have $g_{\bf b} = \sum_{u\in M} g_u^{(\bf b)}x^{-u}$, where $g_u^{(\bf b)} = \prod_{j=1}^m \big([-v_j]_{s_j}\ell_j^{s_j}\big)$ if there exist $\{s_j\}_{j=1}^m\subseteq{\mathbb Z}_{\leq 0}$, necessarily unique, such that $-u = \sum_{j=1}^m (-v_j + s_j){\bf a}_j$ and $g_u^{(\bf b)} = 0$ otherwise.
\end{lemma}

\begin{proposition}
The set $\{g_{\bf b}\}_{{\bf b}\in{\mathcal B}}$ is the basis of ${\rm Hom}_{\mathbb C}({\mathcal W}_0,{\mathbb C})$ dual to the basis $X^{\mathcal B}$ of ${\mathcal W}_0$.
\end{proposition}

\begin{proof}
We first check that the $g_{\bf b}$ lie in ${\rm Hom}_{\mathbb C}({\mathcal W}_0,{\mathbb C})$.
By Proposition 2.9 we need to check that $D^*_{i,0}(g_{\bf b}) = 0$ for $i=1,\dots,n$.  First note that 
\begin{align*}
x_i\frac{\partial g_{\bf b}}{\partial x_i} &= \sum_{k=1}^m \bigg(\prod_{\substack{j=1\\ j\neq k}}^m \ell_j^{v_j}g_j(\ell_jx^{{\bf a}_j})\bigg)  x_i\frac{\partial}{\partial x_i}\big(\ell_k^{v_k}g_k(\ell_kx^{{\bf a}_k})\big)\\
 &= \sum_{k=1}^m \bigg(\prod_{\substack{j=1\\ j\neq k}}^m \ell_j^{v_j}g_j(\ell_jx^{{\bf a}_j})\bigg)(\ell_k^{v_k+1}a_{ik}x^{{\bf a}_k})g_k'(\ell_kx^{{\bf a}_k}).
\end{align*}
Since $g_k'(t) = g_k(t) -t^{-v_j}$, this reduces to
\begin{equation}
x_i\frac{\partial g_{\bf b}}{\partial x_i} = \sum_{k=1}^m \bigg(\prod_{\substack{j=1\\ j\neq k}}^m \ell_j^{v_j}g_j(\ell_jx^{{\bf a}_j})\bigg)
\big(\ell_k^{v_k+1}a_{ik}x^{{\bf a}_k}g_k(\ell_kx^{{\bf a}_k}) - a_{ik}\ell_kx^{(-v_k+1){\bf a}_k}\big).
\end{equation}
Because $(-v_k+1)>0$, no term in the product $x^{(-v_k+1){\bf a}_k}\prod_{\substack{j=1\\ j\neq k}}^m g_j(\ell_jx^{{\bf a}_j})$ contains a term $x^{-u}$ with $u\in M$.  It follows that
\begin{align*} 
\gamma_- \bigg(x_i\frac{\partial g_{\bf b}}{\partial x_i}\bigg) &=  \sum_{k=1}^m \bigg(\prod_{\substack{j=1\\ j\neq k}}^m \ell_j^{v_j}g_j(\ell_jx^{{\bf a}_j})\bigg)
\big(\ell_k^{v_k+1}a_{ik}x^{{\bf a}_k}g_k(\ell_kx^{{\bf a}_k}) \big) \\ 
&= \gamma_-\bigg(\sum_{k=1}^m \ell_k a_{ik}x^{{\bf a}_k} g_{\bf b}\bigg) = \gamma_-\bigg(x_i\frac{\partial f_0}{\partial x_i}g_{\bf b}\bigg).
\end{align*}
This equation implies that $D^*_{i,0}(g_{\bf b})=0$ for all $i$.

We also need to check that if ${\bf b},{\bf b}'\in{\mathcal B}$, then
\[ \langle g_{\bf b},x^{{\bf b}'}\rangle = \begin{cases} 1 & \text{if ${\bf b} = {\bf b}'$,} \\ 0 & \text{if ${\bf b}\neq {\bf b}'$.}
\end{cases} \]
It follows from Lemma 3.14 that
\[ \langle g_{\bf b},x^{{\bf b}'}\rangle = g^{({\bf b})}_{{\bf b}'} =\prod_{j=1}^m \big([-v_j]_{s_j}\ell_j^{s_j}\big), \]
where the $v_j$ are determined by the condition $\sum_{j=1}^m v_j{\bf a}_j = {\bf b}$ and the $s_j$ are determined by the condition $-{\bf b}' = \sum_{j=1}^m (-v_j + s_j){\bf a}_j$.  If the second condition has a solution $\{s_j\}_{j=1}^m\subseteq{\mathbb Z}_{\leq 0}$, then ${\bf b}\equiv {\bf b}'\pmod{{\mathbb Z}A}$.  But this can happen only if ${\bf b} = {\bf b}'$, so $g^{({\bf b})}_{{\bf b}'} = 0$ if ${\bf b}\neq{\bf b}'$.  And if ${\bf b}={\bf b}'$, then $s_j=0$ for all $j$ so $g^{({\bf b})}_{{\bf b}'} = 1$.  
\end{proof}

We can now apply the isomorphism of Theorem 2.1.  Define
 \begin{equation}
 \xi_{\bf b} = \phi(g_{\bf b}) = \gamma_-\big(\exp(-\ell_0\lambda x^{{\bf a}_0})g_{\bf b}\big).
 \end{equation}
By Theorem 2.1, we have the following result.
\begin{corollary}
The set $\{\xi_{\bf b}\}_{{\bf b}\in{\mathcal B}}$ is a basis for ${\rm Hom}_{\mathcal D}({\mathcal W},{\mathbb C}[[\lambda]])$.
\end{corollary}

Define $G^{(\bf b)}_u(\lambda)$ by the formula
\begin{equation}
 \xi_{\bf b} = \sum_{u\in M} G^{(\bf b)}_u(\lambda) x^{-u}. 
 \end{equation}
From Lemma 3.14, we have the following explicit formula, which will be simplified later.
\begin{align}
G^{(\bf b)}_u(\lambda)  &= \sum_{\substack{w\in M,\,s\in{\mathbb Z}_{\geq 0}\\-w + s{\bf a}_0 = -u}} g^{(\bf b)}_w\frac{(-\ell_0\lambda)^s}{s!} \\
 &= \sum_{\substack{s_0\in{\mathbb Z}_{\geq 0},\,s_1,\dots,s_m\in{\mathbb Z}_{\leq 0}\\ s_0{\bf a}_0 + \sum_{j=1}^m (-v_j+s_j){\bf a}_j = -u}} \bigg(\prod_{j=1}^m [-v_j]_{s_j}\prod_{j=1}^m \ell_j^{s_j}\bigg) \frac{(-\ell_0)^{s_0}}{s_0!} \lambda^{s_0}. \nonumber
 \end{align}
 
 \section{Structure of solutions}

In this section we analyze Equation (3.20) to describe the $\xi_{\bf b}$ in more detail.  We first make a general observation.

There is a direct sum decomposition of ${\mathcal W}$ as ${\mathcal D}$-module.  Put $N=[V_{\mathbb Z}:{\mathbb Z}A_+]$.  We have
\[ N=\frac{[V_{\mathbb Z}:{\mathbb Z}A]}{[{\mathbb Z}A_+:{\mathbb Z}A]} = \frac{d}{\ell_0}. \]
Let ${\mathcal C}_k$, $k=1,\dots,N$ be the cosets of ${\mathbb Z}A_+$ in $V_{\mathbb Z}$ and put $M_k=M\cap {\mathcal C}_k$.  Let $S_k$ be the ${\mathbb C}[\lambda]$-module generated by the $x^u$ with $u\in M_k$.  The $S_k$ are disjoint, stable under the action of the $D_i$ and $D_\lambda$, and their union is $S$.  Put ${\mathcal W}_k = S_k/\sum_{i=1}^n D_iS_k$.  
We have the following obvious result and corollaries.
\begin{proposition}
There is a direct sum decomposition ${\mathcal W}\cong\bigoplus_{k=1}^N{\mathcal W}_k$.
\end{proposition}
Let ${\mathcal F}$ be a left ${\mathcal D}$-module.
\begin{corollary}
There is a direct sum decomposiiton 
\[ {\rm Hom}_{{\mathcal D}}({\mathcal W},{\mathcal F})\cong\bigoplus_{k=1}^N{\rm Hom}_{\mathcal D}({\mathcal W}_k,{\mathcal F}). \]
\end{corollary}
Put
\begin{multline*} 
{\mathcal K}_k({\mathcal F}) = \bigg\{ \xi = \sum_{u\in M_k} A_ux^{-u}\mid \text{$A_u\in{\mathcal F}$ for $u\in M_k$, $D_i^*(\xi) = 0$ for $i=1,\dots,n$,} \\
 \text{ and $\partial_\lambda(A_u) =-\ell_0A_{u+{\bf a}_0}$ for all $u\in M_k$}\bigg\}. 
 \end{multline*}
\begin{corollary}
We have ${\rm Hom}_{\mathcal D}({\mathcal W}_k,{\mathcal F})\cong {\mathcal K}_k({\mathcal F})$.
\end{corollary}
Let ${\mathcal B}_k = {\mathcal B}\cap{\mathcal C}_k$.  Note that $\lvert{\mathcal B}_k\rvert = d/N = \ell_0$ for $k=1,\dots,N$.
\begin{corollary}
The set $\{\xi_{\bf b}\}_{{\bf b}\in{\mathcal B}_k}$ is a basis for ${\rm Hom}_{\mathcal D}({\mathcal W}_k,{\mathbb C}[[\lambda]])$.
\end{corollary}

 The condition on the second sum in Equation (3.20) shows that if ${\bf b}\in{\mathcal B}_k$, then
\begin{equation}
G^{({\bf b})}_u(\lambda) = 0 \quad\text{if $u\not\in M_k$.}
\end{equation}
Equation (3.19) thus becomes, for ${\bf b}\in{\mathcal B}_k$,
\begin{equation}
\xi_{\bf b} = \sum_{u\in M_k}G^{({\bf b})}_u(\lambda)x^{-u}.
\end{equation}

The following theorem connects the $\xi_{\bf b}$ with classical hypergeometric series.  
\begin{theorem}
Let ${\bf b}\in{\mathcal B}_k$ and $u\in M_k$.  Write ${\bf b} = \sum_{j=1}^m v_j{\bf a}_j$.  Let $s_0\in\{ 0,1,\dots,\ell_0-1\}$ be the unique element satisfying $u+s_0{\bf a}_0\equiv{\bf b}\pmod{{\mathbb Z}A}$ and let $s_1,\dots,s_m\in{\mathbb Z}$ be the unique solution to the equation
\begin{equation}
u+s_0{\bf a}_0 = {\bf b}-\sum_{j=1}^m s_j{\bf a}_j.
\end{equation}
Then
\begin{equation}
G_u^{({\bf b})}(\lambda) = \bigg(\prod_{j=1}^m [-v_j]_{s_j}\prod_{j=1}^m \ell_j^{s_j}\bigg) \frac{(-\ell_0)^{s_0}}{s_0!}F_u^{({\bf b})}(\lambda),
\end{equation}
where
\begin{equation}
F_u^{({\bf b})}(\lambda) = \lambda^{s_0} \sum_{s=0}^\infty \frac{\displaystyle \prod_{j=1}^m \bigg(\frac{v_j-s_j}{\ell_j}\bigg)_s\bigg(\frac{v_j-s_j+1}{\ell_j}\bigg)_s\cdots\bigg(\frac{v_j-s_j+\ell_j-1}{\ell_j}\bigg)_s}{\displaystyle \bigg(\frac{s_0+1}{\ell_0}\bigg)_s\bigg(\frac{s_0+2}{\ell_0}\bigg)_s\cdots\bigg(\frac{s_0+\ell_0}{\ell_0}\bigg)_s}\lambda^{s\ell_0}.
\end{equation}
\end{theorem}

\begin{proof}
The condition on the second summation in (3.20) is equivalent to (4.8).
The solution $\{s_j\}_{j=1}^m$ of (4.8) must satisfy $s_j\leq 0$ for $j=1,\dots,m$ because the left-hand side of (4.8) lies in $C(A)$ while the right-hand side of (4.8) equals $\sum_{j=1}^m (v_j-s_j){\bf a}_j$ and $v_j\in[0,1)$ for all $j$.

It now follows from (1.1) that all solutions of the condition on the summation in (3.20) are of the form
\begin{equation}
(s_0 +s\ell_0){\bf a}_0 -{\bf b} + \sum_{j=1}^m (s_j -s\ell_j){\bf a}_j = -u\quad\text{with $s\in{\mathbb Z}_{\geq 0}$.}
\end{equation}
We can thus write (3.20) as
\begin{equation}
G^{({\bf b})}_u(\lambda) = \lambda^{s_0} \sum_{s=0}^\infty \bigg(\prod_{j=1}^m [-v_j]_{s_j-s\ell_j}\prod_{j=1}^m \ell_j^{s_j-s\ell_j}\bigg) \frac{(-\ell_0)^{s_0+s\ell_0}}{(s_0 + s\ell_0)!}\lambda^{s\ell_0}.
\end{equation}

We simplify (4.12) by factoring out terms that do not depend on $s$.  We have
\[ [-v_j]_{s_j -s\ell_j} = [-v_j]_{s_j}[-v_j+s_j]_{-s\ell_j} \]
and
\[ (s_0+s\ell_0)! = s_0!(s_0+1)_{s\ell_0}. \]
We thus have
\begin{equation}
G_u^{({\bf b})}(\lambda) = \bigg(\prod_{j=1}^m [-v_j]_{s_j}\prod_{j=1}^m \ell_j^{s_j}\bigg) \frac{(-\ell_0)^{s_0}}{s_0!}F_u^{({\bf b})}(\lambda),
\end{equation}
where
\begin{equation}
F_u^{({\bf b})}(\lambda) = \lambda^{s_0}\sum_{s=0}^\infty \bigg(\prod_{j=1}^m[-v_j+s_j]_{-s\ell_j}\prod_{j=1}^m \ell_j^{-s\ell_j}\bigg) \frac{(-\ell_0)^{s\ell_0}}{(s_0+1)_{s\ell_0}}\lambda^{s\ell_0}.
\end{equation}
Using (3.10) we can rewrite this in terms of Pochhammer symbols:
\begin{equation}
F_u^{({\bf b})}(\lambda) = \lambda^{s_0} \sum_{s=0}^\infty \bigg(\prod_{j=1}^m \frac{(v_j-s_j)_{s\ell_j}}{\ell_j^{s\ell_j}}\bigg) \frac{\ell_0^{s\ell_0}}{(s_0+1)_{s\ell_0}}\lambda^{s\ell_0},
\end{equation}
where we have used (1.2) to eliminate minus signs.  Rearranging the right-hand side of (4.15) gives (4.10).
\end{proof}

If ${\bf b}\in{\mathcal B}_k$ is an interior point of $P(A)$, then $v_j>0$ for all $j$, so (4.9) and (4.10) imply that $G^{({\bf b})}_u(\lambda)\neq 0$ for all $u\in M_k$.  But if ${\bf b}$ is a boundary point of $P(A)$ many of the $G^{({\bf b})}_u(\lambda)$ vanish.  

Let ${\bf b}\in{\mathcal B}_k$ and let $\sigma_{\bf b}$ be the smallest closed face of $C(A)$ containing ${\bf b}$.  Denote by $\sigma_{\bf b}^\circ$ the relative interior of $\sigma_{\bf b}$, i.~e., $\sigma_{\bf b}^\circ$ equals $\sigma_{\bf b}$ minus all its proper closed subfaces.  Note that ${\bf b}\in\sigma_{\bf b}^\circ$.  Since ${\bf b} = \sum_{j=1}^m v_j{\bf a}_j$, we have $v_j\neq 0$ if and only if ${\bf a}_j\in\sigma_{\bf b}$.  

\begin{corollary}
Let ${\bf b}\in{\mathcal B}_k$ and $u\in M_k$.  Suppose that $\sigma_{\bf b}\neq C(A)$.  Then $G^{({\bf b})}_u(\lambda)$ is a nonzero constant if and only if $u\in \sigma_{\bf b}^\circ\cap ({\bf b} + {\mathbb Z}_{\geq 0}A)$, i.~e.,
\begin{equation}
u = {\bf b} + \sum_{\{j:{\bf a}_j\in\sigma_{\bf b}\}}t_j{\bf a}_j\quad\text{for (unique) $t_j\in{\mathbb Z}_{\geq 0}$}.
\end{equation}
When $(4.17)$ holds, one has
\begin{equation}
G^{({\bf b})}_u(\lambda) = \bigg( \prod_{\{j:{\bf a}_j\in\sigma_{\bf b}\}} \ell_j^{-t_j}\bigg)\bigg(\prod_{\{j:{\bf a}_j\in\sigma_{\bf b}\}} [-v_j]_{-t_j}\bigg).
\end{equation}
If $u$ is not of the form $(4.17)$, then $G^{({\bf b})}_u(\lambda) = 0$.
\end{corollary}

\begin{proof}
Since $\sigma_{\bf b}\neq C(A)$, we must have $v_j=0$ for some $j$.  To fix ideas, suppose that $v_j\neq 0$ for $j=1,\dots,J$ and $v_j=0$ for $j=J+1,\dots,m$, where $J<m$.  Equation~(4.8) gives
\begin{equation}
u+s_0{\bf a}_0 = \sum_{j=1}^m (v_j-s_j){\bf a}_j. 
\end{equation}
If $s_j<0$ for some $j\in\{J+1,\dots,m\}$, then the factor $[-v_j]_{s_j} = [0]_{s_j} = 0$ appears on the right-hand side of (4.13), so $G^{({\bf b})}_u(\lambda)=0$.  So we assume $s_j=0$ for $j=J+1,\dots,m$ and (4.19) becomes
\begin{equation}
u+s_0{\bf a}_0 = \sum_{j=1}^J (v_j-s_j){\bf a}_j. 
\end{equation}
If $s_0>0$, then the left-hand side of (4.20) is an interior point of $C(A)$ while the right-hand side of (4.20) lies in $\sigma_{\bf b}^\circ$, so we must have $s_0=0$ and (4.20) becomes
\begin{equation}
u = {\bf b} -\sum_{j=1}^J s_j{\bf a}_j.
\end{equation}
This proves that $G^{({\bf b})}_u(\lambda)=0$ if $u$ is not of the form (4.17).  

Equation (4.21) gives $v_j=s_j=0$ for $j=J+1,\dots,m$, so (4.10) implies that $F^{({\bf b})}_u(\lambda) = 1$.  Equation~(4.18) then follows from (4.13).  
\end{proof}

\begin{corollary}
Suppose in addition to the hypotheses of Corollary $4.16$ that $u\in{\mathcal B}_k$.  Then
\begin{equation}
G^{({\bf b})}_u(\lambda) = \begin{cases} 1 & \text{if $u={\bf b}$,} \\ 0 & \text{if $u\neq {\bf b}$.}\end{cases}
\end{equation}
\end{corollary}

\begin{proof}
Since $u,{\bf b}\in{\mathcal B}_k$, Equation (4.17) has a solution if and only if $u = {\bf b}$, in which case all $t_j$ equal 0.  The assertion then follows from Corollary~4.16.
\end{proof}

Since $s_0\in\{0,1,\dots,\ell_0-1\}$, there are only finitely many possibilities for the denominator on the right-hand side of (4.10).    It follows that for all but finitely many $u$ there is no cancellation in the quotient on the right-hand side of (4.10).  But for $u\in{\mathcal B}_k$ there is often cancellation, which we describe precisely in what follows.

\section{The differential module ${\mathcal M}'$}

The main purpose of this section is to show that the set $X^{\mathcal B}$, which is a ${\mathbb C}$-basis for ${\mathcal W}_0$ by Proposition~3.9, is also a ${\mathbb C}(\lambda)$-basis for ${\mathcal M}'$.

Let $\Delta(A)\subseteq V$ be the convex hull of $A\cup\{{\bf 0}\}$.  For each face $\sigma$ of $\Delta(A)$ not containing ${\bf a}_0$, let $f_{{\lambda}}^{\sigma} = \sum_{{\bf a}_j\in\sigma} \ell_jx^{{\bf a}_j}$, a Laurent polynomial not involving $\lambda$.  
If $\sigma$ contains ${\bf a}_0$,  we define $f_{{\lambda}}^{\sigma} = f_{{\lambda}}$.  Recall that $f_{{\lambda}}$ is {\it nondegenerate relative to $\Delta(A)$\/} if for each face $\sigma$ of $\Delta(A)$ that does not contain the origin, the polynomials $\{x_i\partial f^\sigma_{{\lambda}}/\partial x_i\}_{i=1}^n$ have no common zero in $\big(\overline{{\mathbb C}(\lambda)}^\times\big)^n$, where $\overline{{\mathbb C}(\lambda)}$ denotes an algebraic closure of~${\mathbb C}(\lambda)$.

\begin{proposition}
The Laurent polynomial $f_{{\lambda}}$ is nondegenerate relative to $\Delta(A)$.
\end{proposition}

\begin{proof}
Suppose first that $\sigma$ is the (unique) face of $\Delta(A)$ not containing the origin but containing~${\bf a}_0$.  Then
\[ x_i\frac{\partial f^\sigma_{{\lambda}}}{\partial x_i} = \sum_{j=1}^m \ell_ja_{ij}x^{{\bf a}_j} - \ell_0a_{i0}{\lambda} x^{{\bf a}_0} \]
for $i=1,\dots,n$.  By (1.1), this can be rewritten as
\[ x_i\frac{\partial f^\sigma_{{\lambda}}}{\partial x_i} = \sum_{j=1}^m \ell_ja_{ij}(x^{{\bf a}_j} - {\lambda} x^{{\bf a}_0}), \]
or, in matrix form,
\begin{equation}
\begin{pmatrix} \displaystyle x_1\frac{\partial f_{{\lambda}}^\sigma}{\partial x_1} \\ \vdots \\ \displaystyle x_n\frac{\partial f_{{\lambda}}^\sigma}{\partial x_n} \end{pmatrix}  = \boldsymbol{\alpha}
\begin{pmatrix} \ell_1(x^{{\bf a}_1}-{\lambda} x^{{\bf a}_0}) \\ \vdots \\ \ell_m(x^{{\bf a}_m}-{\lambda} x^{{\bf a}_0}) \end{pmatrix},
\end{equation}
where $\boldsymbol{\alpha}$ is the matrix defined in Section 3.  Let $\boldsymbol{\beta}$ also be as defined in Section~3.
Multiplying both sides of (5.2) on the left by $\boldsymbol{\beta}$ gives
\begin{equation}
\boldsymbol{\beta} \begin{pmatrix} \displaystyle x_1\frac{\partial f_{{\lambda}}^\sigma}{\partial x_1} \\ \vdots \\ \displaystyle x_n\frac{\partial f_{{\lambda}}^\sigma}{\partial x_n} \end{pmatrix}  = \begin{pmatrix} \ell_1(x^{{\bf a}_1}-{\lambda} x^{{\bf a}_0}) \\ \vdots \\ \ell_m(x^{{\bf a}_m}-{\lambda} x^{{\bf a}_0})\\ 0 \\ \vdots \\ 0 \end{pmatrix}.
\end{equation}
Since $\boldsymbol{\beta}$ is invertible, this implies that the $x_i\partial f_{{\lambda}}^\sigma/\partial x_i$ have a common zero in $\big(\overline{{\mathbb C}(\lambda)}^\times\big)^n$ if and only if the $x^{{\bf a}_j}-{\lambda}x^{{\bf a}_0}$ have a common zero in $\big(\overline{{\mathbb C}(\lambda)}^\times\big)^n$.  But if $x^{{\bf a}_j} = {\lambda}x^{{\bf a}_0}$ for $j=1,\dots,m$, then
\[ \prod_{j=1}^m x^{\ell_j{\bf a}_j} = \prod_{j=1}^m {\lambda}^{\ell_j}x^{\ell_j{\bf a}_0}. \]
Using (1.1) and (1.2), this reduces to the equation $(1-\lambda^{\ell_0})x^{\ell_0{\bf a}_0} = 0$.  Since the $x_i$ lie in $\overline{{\mathbb C}(\lambda)}^\times$, it is not possible to satisfy this equation.

Now suppose $\sigma$ is a face of $\Delta(A)$ containing neither the origin nor ${\bf a}_0$.  Arguing as above, one obtains an equation analogous to (5.3) but with the nonzero terms in the column on the right-hand side being $\{\ell_jx^{{\bf a}_j}\}_{{\bf a}_j\in\sigma}$.  This shows that the $x_i\partial f_{{\lambda}}^\sigma/\partial x_i$ have a common zero in $\big(\overline{{\mathbb C}(\lambda)}^\times\big)^n$ if and only if the $x^{{\bf a}_j}$ for ${\bf a}_j\in\sigma$ have a common zero in $\big(\overline{{\mathbb C}(\lambda)}^\times\big)^n$.  It is clear that no such zero exists, thus $f_{{\lambda}}$ is nondegenerate.  
\end{proof}

Let ${\rm vol}(\Delta(A))$ be the volume of $\Delta(A)$ relative to Lebesgue measure on $V$ normalized by requiring the volume of a fundamental domain for $V_{\mathbb Z}$ to be 1.  Thus
\[ m!{\rm vol}(\Delta(A)) = {\rm vol}(P(A)) = [V_{\mathbb Z}:{\mathbb Z}A] = d. \]
From a result of Kouchnirenko \cite{Ko} (see also \cite[Theorem~2.17]{AS0}) we get the following corollary.
\begin{corollary}
One has $\dim_{{\mathbb C}(\lambda)}S'\big/\sum_{i=1}^n \big(x_i\frac{\partial f_{{\lambda}}}{\partial x_i}\big)S' = d$.
\end{corollary}

\begin{proposition}
The set $X^{\mathcal B}$ is a basis for $S'\big/\sum_{i=1}^n \big(x_i\frac{\partial f_{{\lambda}}}{\partial x_i}\big)S'$.  
\end{proposition}

\begin{proof}
By Corollary 5.4  the set $X^{\mathcal B}$ has the correct cardinality, so it suffices to prove that it spans the quotient.  Since 
$\boldsymbol{\beta}$ is invertible, Equation (5.3) shows that
\begin{equation}
S'\big/\sum_{i=1}^n \big(x_i\frac{\partial f_{{\lambda}}}{\partial x_i}\big)S' = S'\big/ \sum_{j=1}^m (x^{{\bf a}_j}-{\lambda}x^{{\bf a}_0})S'.
\end{equation}
Let $u_0\in M$ and consider $x^{u_0}\in S'$.  If $u_0\in{\mathcal B}$ then $x^{u_0}\in X^{\mathcal B}$, so there is nothing to prove.  Suppose $u_0\not\in{\mathcal B}$.  Then there exist $j_1\in\{1,\dots,m\}$ and $u_1\in M$ such that 
$u_0 = {\bf a}_{j_1} +  u_1$, hence
\[ x^{u_0} = (x^{{\bf a}_{j_1}}-{\lambda}x^{{\bf a}_0})x^{u_1} +{\lambda} x^{u_1 + {\bf a}_0}. \]
If $u_1 + {\bf a}_0\in{\mathcal B}$, then $x^{u_1+{\bf a}_0}\in X^{\mathcal B}$ and we are done.  Otherwise there exist $j_2\in\{1,\dots,m\}$ and $u_2\in M$ such that $u_1+{\bf a}_0 = {\bf a}_{j_2} + u_2$, hence
\[ {\lambda}x^{u_1 + {\bf a}_0} = (x^{{\bf a}_{j_2}} -{\lambda}x^{{\bf a}_0}){\lambda}x^{u_2} + {\lambda}^2 x^{u_2 + {\bf a}_0}. \]
If $u_2+{\bf a}_0\in{\mathcal B}$, we are done, otherwise this procedure can be continued.  

If this procedure continues for $K$ steps with $u_k+{\bf a}_0\not\in{\mathcal B}$ for $k=0,1,\dots,K-1$, we get a sequence $u_0,u_1,\dots,u_K\in M$ with
\begin{equation}
 {\lambda}^{k-1} x^{u_{k-1} + {\bf a}_0} = (x^{{\bf a}_{j_k}} -{\lambda}x^{{\bf a}_0}){\lambda}^{k-1} x^{u_k} + {\lambda}^k x^{u_k + {\bf a}_0}.
 \end{equation}
Note that in this procedure we have inductively
\begin{equation}
x^{u_0}\equiv {\lambda}^kx^{u_k+{\bf a}_0}\pmod{\sum_{j=1}^m (x^{{\bf a}_j}-{\lambda}x^{{\bf a}_0})S'}
\end{equation}
and
\begin{equation}
u_k+{\bf a}_0= u_0 + k{\bf a}_0 - {\bf a}_{j_1} - \cdots - {\bf a}_{j_k}.
\end{equation}

If this process terminates for $K<\ell_0$, we are done.  If we can continue to $K=\ell_0 $, then Equations~(5.8) and~(5.9) imply that
\begin{equation}
x^{u_0}\equiv {\lambda}^{\ell_0}x^{u_{\ell_0}+{\bf a}_0}\pmod{\sum_{j=1}^m (x^{{\bf a}_j}-{\lambda}x^{{\bf a}_0})S'}
\end{equation}
and
\begin{equation}
u_{\ell_0}+{\bf a}_0= u_0 + \ell_0{\bf a}_0 - {\bf a}_{j_1} - \cdots - {\bf a}_{j_{\ell_0}}.
\end{equation}

Note that for any pair $j,j'\in\{1,\dots,m\}$ we have 
\[ x^{{\bf a}_j}-x^{{\bf a}_{j'}} = (x^{{\bf a}_j}-\lambda x^{{\bf a}_0}) - (x^{{\bf a}_{j'}}-\lambda x^{{\bf a}_0})\in \sum_{j=1}^m (x^{{\bf a}_j}-{\lambda}x^{{\bf a}_0})S'. \]
 So if $u-{\bf a}_j \in M$, we have
\[ x^{u-{\bf a}_j}(x^{{\bf a}_j} - x^{{\bf a}_{j'}}) = x^u-x^{u+{\bf a}_{j'}-{\bf a}_j}, \]
hence
\begin{equation}
x^u\equiv x^{u+{\bf a}_{j'}-{\bf a}_j} \pmod{\sum_{j=1}^m (x^{{\bf a}_j}-{\lambda}x^{{\bf a}_0})S'}.
\end{equation}
Applying this repeatedly to (5.11) we get
\begin{equation}
x^{u_{\ell_0} + {\bf a}_0}\equiv x^{u_0+ \ell_0{\bf a}_0 -{\bf a}_{j_1'} - \cdots - {\bf a}_{j_{\ell_0}'}} 
\pmod{\sum_{j=1}^m (x^{{\bf a}_j}-{\lambda}x^{{\bf a}_0})S'}
\end{equation}
for any sequence ${\bf a}_{j_1'},\dots,{\bf a}_{j_{\ell_0}'}$.  We now choose this sequence to be ${\bf a}_1$ taken $\ell_1$ times, ${\bf a}_2$ taken $\ell_2$ times, \dots, ${\bf a}_m$ taken $\ell_m$ times.  Equations (1.1) and (5.13) now imply
\begin{equation}
x^{u_{l_0} + {\bf a}_0}\equiv x^{u_0} \pmod{\sum_{j=1}^m (x^{{\bf a}_j}-{\lambda}x^{{\bf a}_0})S'}.
\end{equation}
From Equation (5.10) we get
\begin{equation}
({\lambda}^{\ell_0} - 1)x^{u_0}\equiv 0 \pmod{\sum_{j=1}^m (x^{{\bf a}_j}-{\lambda}x^{{\bf a}_0})S'},
\end{equation}
which completes the proof.
\end{proof}

We record for possible future use a corollary of the proof, which can be obtained by arguing with $S$ in place of $S'$.
\begin{corollary}
We have
\[ (1-\lambda^{\ell_0})S\subseteq \sum_{{\bf b}\in{\mathcal B}} {\mathbb C}[\lambda]x^{{\bf b}} + \sum_{j=1}^m (x^{{\bf a}_j} -\lambda x^{{\bf a}_0})S. \]
\end{corollary}

Relative to the grading defined on $S'$ in Section 3, the elements $x^{{\bf a}_j}-\lambda x^{{\bf a}_0}$ are homogeneous 
of degree~1.  To show that these elements form a regular sequence on $S'$ it suffices to show that if $\{g_j(\lambda,x)\}_{j=1}^m$ are elements of $S'$ such that
\begin{equation}
 \sum_{j=1}^m (x^{{\bf a}_j}-\lambda x^{{\bf a}_0})g_j(\lambda,x) = 0, 
 \end{equation}
then there exists a skew-symmetric set $\{h_{ij}(\lambda,x)\}_{i,j=1}^m$ of elements of $S'$ (i.~e., $h_{ij} = -h_{ji}$) such that
\begin{equation}
 g_j(\lambda,x) = \sum_{i=1}^m (x^{{\bf a}_i}-\lambda x^{{\bf a}_0})h_{ij}(\lambda,x) \quad\text{for $j=1,\dots,m$.} 
 \end{equation}

\begin{proposition}
The elements $\{x^{{\bf a}_j} -\lambda x^{{\bf a}_0}\}_{j=1}^m$ form a regular sequence on $S'$.
\end{proposition}

\begin{proof}
Since $S'$ is a graded ring and the $x^{{\bf a}_j}-\lambda x^{{\bf a}_0}$ are homogeneous, it suffices to prove that (5.18) follows from (5.17) when the $g_j(\lambda,x)$ are homogeneous of the same degree.  After multiplying (5.17) by an element of ${\mathbb C}[\lambda]$ we may assume that the $g_j(\lambda,x)$ lie in $S$.  As such they may be written in the form
\begin{equation}
g_j(\lambda,x)  = \sum_{k=0}^N \lambda^kg_j^{(k)}(x)
\end{equation}
where the $g_j^{(k)}(x)$ are homogeneous elements of the same degree in $S_0$.  Equation~(5.17) then implies that
\[ \sum_{j=1}^m x^{{\bf a}_j}g_j^{(0)}(x) = 0. \]
Since the $x^{{\bf a}_j}$ form a regular sequence in $S_0$, there exist $h_{ij}^{(0)}(x)$ in $S_0$, skew-symmetric, such that 
\begin{equation}
 g_j^{(0)}(x) = \sum_{i=1}^m x^{{\bf a}_i}h^{(0)}_{ij}(x). 
 \end{equation}
Define
\[ \tilde{g}_j(\lambda,x) =g_j(\lambda,x) - \sum_{i=1}^m (x^{{\bf a}_i}-\lambda x^{{\bf a}_0})h^{(0)}_{ij}(x). \]
By (5.17) and skew-symmetry, we have
\[ \sum_{j=1}^m (x^{{\bf a}_j}-\lambda x^{{\bf a}_0})\tilde{g}_j(\lambda,x) = 0. \]
Furthermore, from (5.20) and (5.21) we have
\[ \lambda^{-1}\tilde{g}_j(\lambda,x) =  x^{{\bf a}_0}\sum_{i=1}^m h_{ij}^{(0)}(x) + \sum_{k=0}^{N-1} \lambda^k g_j^{(k+1)}(x). \]
It follows that
\[ \sum_{j=1}^m (x^{{\bf a}_j}-\lambda x^{{\bf a}_0})\big(\lambda^{-1}\tilde{g}_j(\lambda,x)\big) = 0, \]
hence it suffices to prove that (5.17) implies (5.18) when the $g_j(\lambda,x)$ are replaced by the $\lambda^{-1}\tilde{g}_j(\lambda,x)$.  Since the $\lambda^{-1}\tilde{g}_j(\lambda,x)$ are of degree $N-1$ in $\lambda$, we are reduced by induction to the case where $N=0$, i.~e., the $g_j(\lambda,x)$ do not involve $\lambda$.  

So we now suppose that we have elements $g_j(x)$ of $S_0$, homogeneous of the same degree, such that
\begin{equation}
\sum_{j=1}^m (x^{{\bf a}_j}-\lambda x^{{\bf a}_0})g_j(x) = 0 
 \end{equation}
in $S'$.  This implies that
\begin{equation}
 \sum_{j=1}^m x^{{\bf a}_j}g_j(x) = 0
 \end{equation}
 and
 \begin{equation}
 \sum_{j=1}^m g_j(x) = 0.
 \end{equation}
Since the $x^{{\bf a}_j}$ form a regular sequence in $S_0$ there exists a skew-symmetric set $\{h_{ij}(x)\}_{i,j=1}^m$ of elements of $S_0$ with $\deg h_{ij}=\deg g_j-1$ such that
\begin{equation}
 g_j(x) = \sum_{i=1}^m x^{{\bf a}_i}h_{ij}(x). 
 \end{equation}
Define 
\begin{equation}
\tilde{g}_j(x) = g_j(x) - \sum_{i=1}^m (x^{{\bf a}_i} - \lambda x^{{\bf a}_0})h_{ij}(x).
\end{equation}
By (5.22) and skew-symmetry we have
\begin{equation}
\sum_{j=1}^m (x^{{\bf a}_j}-\lambda x^{{\bf a}_0})\tilde{g}_j(x) = 0.
\end{equation}
Furthermore, from (5.25), we have 
\[ \tilde{g}_j(x) = \lambda x^{{\bf a}_0}\sum_{i=1}^m h_{ij}(x), \]
so (5.27) implies
\[ \sum_{j=1}^m (x^{{\bf a}_j}-\lambda x^{{\bf a}_0})\bigg(\sum_{i=1}^m h_{ij}(x)\bigg) = 0. \]
Since the $h_{ij}(x)$ have degree one less than the $g_j(x)$, we are reduced by induction to the case where the $g_j(x)$ have degree less than $1$.  In this case Equation~(5.23) implies that the $g_j(x)$ equal~0 because the $x^{{\bf a}_j}$ form a regular sequence on $S_0$.  
\end{proof}

We now follow the argument of Section 3.  In analogy with (3.2) we define differential operators $\tilde{D}_i$, $i=1,\dots,n$ by the formula
\begin{equation}
\begin{pmatrix} \tilde{D}_{1} \\ \vdots \\ \tilde{D}_{n} \end{pmatrix} = \boldsymbol{\beta}\begin{pmatrix} {D}_{1} \\ \vdots \\ {D}_{n} \end{pmatrix} = \begin{pmatrix} \sum_{j=1}^n \beta_{1j}D_{j} \\ \vdots \\ \sum_{j=1}^n \beta_{nj}D_{j} \end{pmatrix}.
\end{equation}
\begin{lemma}
Let $u\in M$ and write $u =(u_1,\dots,u_n) = \sum_{j=1}^m u'_j{\bf a}_j$.  We have
\[ \tilde{D}_{j}(x^u) = \begin{cases} u'_jx^u + \ell_j x^u(x^{{\bf a}_j}-\lambda x^{{\bf a}_0}) & \text{for $j=1,\dots,m$,} \\
0 & \text{for $j=m+1,\dots,n$.} \end{cases} \]
\end{lemma}

\begin{proof}
We note that
\begin{equation}
 \begin{pmatrix} x_1\frac{\partial f_\lambda}{\partial x_1} \\ \vdots \\ x_n\frac{\partial f_\lambda}{\partial x_n} \end{pmatrix}  = 
\boldsymbol{\alpha}\begin{pmatrix} \ell_1 (x^{{\bf a}_1}-\lambda x^{{\bf a}_0}) \\ \vdots \\ \ell_m(x^{{\bf a}_m}- \lambda x^{{\bf a}_0}) \end{pmatrix}
\end{equation}
and 
\begin{equation}
\boldsymbol{\alpha}\begin{pmatrix} u'_1 \\ \vdots \\ u'_m \end{pmatrix} = \begin{pmatrix} u_1 \\ \vdots \\ u_n \end{pmatrix}.
\end{equation}
From the definition we have
\begin{align} 
\begin{pmatrix} \tilde{D}_{1} \\ \vdots \\ \tilde{D}_{n} \end{pmatrix}(x^u)  & = \boldsymbol{\beta} \begin{pmatrix} x_1\frac{\partial}{\partial x_1} + x_1 \frac{\partial f_\lambda}{\partial x_1} \\ \vdots \\ x_n\frac{\partial}{\partial x_n} + x_n \frac{\partial f_\lambda}{\partial x_n}\end{pmatrix} (x^u) \\
 &= \boldsymbol{\beta}\begin{pmatrix} u_1 \\ \vdots \\ u_n \end{pmatrix} x^u + \boldsymbol{\beta} \begin{pmatrix}  x_1 \frac{\partial f_\lambda}{\partial x_1} \\ \vdots \\ x_n \frac{\partial f_\lambda}{\partial x_n}\end{pmatrix}x^u \nonumber \\
  &= \boldsymbol{\beta\alpha}\begin{pmatrix} u'_1 \\ \vdots \\ u'_m\end{pmatrix} x^u + \boldsymbol{\beta\alpha}\begin{pmatrix} \ell_1(x^{{\bf a}_1}-\lambda x^{{\bf a}_0}) \\ \vdots \\ \ell_m(x^{{\bf a}_m}-\lambda x^{{\bf a}_0}) \end{pmatrix} x^u  \nonumber \\
  & = \begin{pmatrix} u'_1x^u + \ell_1 x^u(x^{{\bf a}_1}-\lambda x^{{\bf a}_0}) \\ \vdots \\ u'_mx^u + \ell_m x^u(x^{{\bf a}_m}-\lambda x^{{\bf a}_0}) \\ 0 \\ \vdots \\ 0 \end{pmatrix} \nonumber,
 \end{align}
where the third equality follows from (5.30) and (5.31) and the fourth equality follows from (3.1).
\end{proof}

Since $\boldsymbol{\beta}$ is invertible, we get the following result.
\begin{corollary}
We have ${\mathcal W}' = S'/\sum_{j=1}^m \tilde{D}_jS'$.
\end{corollary}

\begin{proposition}
The set $X^{\mathcal B}$ is a basis for ${\mathcal W}'$.
\end{proposition}

\begin{proof}
Consider the Koszul complex ${\rm Kos}^\bullet(S',\{\tilde{D}_{j}\}_{j=1}^m)$ defined by the $\tilde{D}_{j}$.  The grading defined in Section~3 makes it into a filtered complex for which (by Corollary~5.33)
\[ H^m({\rm Kos}^\bullet(S',\{\tilde{D}_{j}\}_{j=1}^m)) = {\mathcal W}'. \]
 The associated graded complex is ${\rm Kos}^\bullet(S',\{\ell_j(x^{{\bf a}_j}-\lambda x^{{\bf a}_0}\}_{j=1}^m)$.  By Proposition~5.19 the set $\{x^{{\bf a}_j}-\lambda x^{{\bf a}_0}\}_{j=1}^m$ is a regular sequence on $S'$, so 
 \[ H^i({\rm Kos}^\bullet(S',\{\ell_j(x^{{\bf a}_j}-\lambda x^{{\bf a}_0}\}_{j=1}^m)) = 0\quad\text{for $0\leq i<m$.} \]
   It follows that 
\[ H^i({\rm Kos}^\bullet(S',\{\tilde{D}_{j}\}_{j=1}^m)) = H^i({\rm Kos}^\bullet(S',\{\ell_j(x^{{\bf a}_j}-\lambda x^{{\bf a}_0}\}_{j=1}^m)) \]
for all $i$.  In particular we have
\begin{align*}
 {\mathcal W}' &= H^m({\rm Kos}^\bullet(S',\{\tilde{D}_{j}\}_{j=1}^m)) = H^m({\rm Kos}^\bullet(S',\{\ell_j(x^{{\bf a}_j}-\lambda x^{{\bf a}_0}\}_{j=1}^m))\\  
 &= S'/\sum_{j=1}^m (x^{{\bf a}_j}-\lambda x^{{\bf a}_0})S'. 
 \end{align*}
 Proposition 5.5 and Equation (5.6) now imply the proposition.
\end{proof}

\section{Fundamental solution matrix}

Associated to the basis $X^{\mathcal B}$ of ${\mathcal W}'$ is a Picard-Fuchs equation described as follows.  Write ${\mathcal B} = \{{\bf b}_1,\dots,{\bf b}_d\}$.  For $i,j=1,\dots,d$, define $C_{ij}(\lambda)\in{\mathbb C}(\lambda)$ by
\begin{equation}
D_\lambda(x^{{\bf b}_i}) = \sum_{j=1}^d C_{ij}(\lambda)x^{{\bf b}_j} 
\end{equation}
in ${\mathcal W}'$.
Let $Y$ be the column vector with entries $y_1,\dots,y_d$ and let $C(\lambda)$ be the $(d\times d)$-matrix $C(\lambda) = (C_{ij}(\lambda))_{i,j=1}^d$.  The associated Picard-Fuchs equation is
\begin{equation}
\partial_\lambda Y = C(\lambda)Y.
\end{equation}
Let $\xi\in{\rm Hom}_{{\mathbb C}(\lambda)}({\mathcal W}',{\mathbb C}((\lambda)))$.  Then $\xi\in {\rm Hom}_{{\mathcal D}'}({\mathcal W}',{\mathbb C}((\lambda)))$ if and only if the column vector with entries $\xi(x^{{\bf b}_1}),\dots,\xi(x^{{\bf b}_d})$ is a solution of the system (6.2).  Since $X^{\mathcal B}$ is a basis for ${\mathcal W}'$, the map that sends $\xi$ to the column vector with entries $\xi(x^{{\bf b}_1}),\dots,\xi(x^{{\bf b}_d})$ identifies the ${\mathbb C}$-vector space ${\rm Hom}_{{\mathcal D}'}({\mathcal W}',{\mathbb C}((\lambda)))$ with the space of solutions of (6.2) that lie in ${\mathbb C}((\lambda))^d$.  This latter space has dimension less than or equal to $d$, so
\[ \dim_{\mathbb C} {\rm Hom}_{{\mathcal D}'}({\mathcal W}',{\mathbb C}((\lambda)))\leq d. \]
But there is a natural inclusion
\[ {\rm Hom}_{{\mathcal D}}({\mathcal W},{\mathbb C}[[\lambda]])\hookrightarrow {\rm Hom}_{{\mathcal D}'}({\mathcal W}',{\mathbb C}((\lambda))), \]
and $\dim_{\mathbb C} {\rm Hom}_{{\mathcal D}}({\mathcal W},{\mathbb C}[[\lambda]]) = d$ by Corollary~3.18.  It follows that (6.2) has a full set of solutions in ${\mathbb C}[[\lambda]]^d$.  Corollary 3.18 implies that the set $\{\xi_{{\bf b}_j}\}_{j=1}^d$ is a basis for ${\rm Hom}_{{\mathcal D}'}({\mathcal W}',{\mathbb C}((\lambda)))$.  The associated fundamental solution matrix of (6.2) is the $(d\times d)$-matrix whose $(i,j)$-entry is, by (3.19),
\begin{equation} 
\xi_{{\bf b}_j}(x^{{\bf b}_i}) = G^{({\bf b}_j)}_{{\bf b}_i}(\lambda).
\end{equation}

As in Section 4, we have a direct sum decomposition of ${\mathcal W}'$.  For $k=1,\dots,N$, let $S'_k$ be the ${\mathbb C}(\lambda)$-vector space generated by the $x^u$ with $u\in M_k$ and put ${\mathcal W}'_k = S'_k/\sum_{i=1}^n D_iS'_k$.  The  set $X^{{\mathcal B}_k}$ is a basis for ${\mathcal W}'_k$ and the set $\{\xi_{\bf b}\}_{{\bf b}\in{\mathcal B}_k}$, with $\xi_{\bf b}$ given by (4.6), is a basis for ${\rm Hom}_{{\mathcal D}'}({\mathcal W}'_k,{\mathbb C}((\lambda)))$.  
Fix $k$ and write ${\mathcal B}_k = \{{\bf b}_1,\dots,{\bf b}_{\ell_0}\}$.  The associated fundamental solution matrix is the $(\ell_0\times\ell_0)$-matrix whose $(i,j)$-entry is $G^{({\bf b}_j)}_{{\bf b}_i}(\lambda)$.  

As we observed in Section 4, if ${\bf b}_j$ is an interior point of $P(A)$, then $G^{({\bf b}_j)}_{{\bf b}_i}(\lambda)\neq 0$ for all ${\bf b}_i\in{\mathcal B}_k$ i.~e., all entries in column $j$ are nonzero.  On the other hand, by Corollary~4.22, if ${\bf b}_j$ is a boundary point of $P(A)$, then 
\begin{equation}
G^{({\bf b}_j)}_{{\bf b}_i}(\lambda) = \begin{cases} 1 & \text{if $i=j$,} \\ 0 & \text{if $i\neq j$.} \end{cases}
\end{equation}
Thus all entries in column $j$ are zero except for the entry in row $j$, which equals 1.

We now describe the entries in each row.  We fix $i$ and let $j$ vary.  By the above, if ${\bf b}_i$ is an interior point of $P(A)$, then $G_{{\bf b}_i}^{({\bf b}_j)}(\lambda)\neq 0$ if and only if ${\bf b}_j$ is an interior point of $P(A)$.  If ${\bf b}_i$ is a boundary point of $P(A)$, then $G_{{\bf b}_i}^{({\bf b}_j)}(\lambda)\neq 0$ if and only if either ${\bf b}_j$ is an interior point of $P(A)$ or ${\bf b}_j = {\bf b}_i$.  Let $R_k$ be the number of elements of ${\mathcal B}_k$ that are interior to $P(A)$.  Then row $i$ has $R_k$ nonzero entries if ${\bf b}_i$ is an interior point of $P(A)$ and has $R_k+1$ nonzero entries if ${\bf b}_i$ is a boundary point of $P(A)$.  

Since $F^{({\bf b}_j)}_{{\bf b}_i}(\lambda)$ determines 
$G^{({\bf b}_j)}_{{\bf b}_i}(\lambda)$ up to a constant factor, we examine the $F^{({\bf b}_j)}_{{\bf b}_i}(\lambda)$ more closely.  For each~$j$, write
\begin{equation}
 {\bf b}_j = \sum_{k=1}^m v_k^{(j)}{\bf a}_k. 
 \end{equation}
Let $s_0^{(j)}\in\{0,1,\dots,\ell_0-1\}$ be the unique element such that
\[ {\bf b}_i + s_0^{(j)}{\bf a}_0\equiv {\bf b}_j\pmod{{\mathbb Z}A}. \]
Then let $s^{(j)}_k$, $k=1,\dots,m,$ be the unique nonpositive integers such that
\begin{equation}
{\bf b}_i + s_0^{(j)}{\bf a}_0 = {\bf b}_j - \sum_{k=1}^m s_k^{(j)}{\bf a}_k.
\end{equation}
The $s_k^{(j)}$ for $j=0,1,\dots,m$ depend on $i$, but since $i$ is fixed this should not cause any confusion.  Note that for fixed $i$, the correspondence ${\bf b}_j\leftrightarrow s_0^{(j)}$ is one-to-one between ${\mathcal B}_k$ and the set $\{0,1,\dots,\ell_0-1\}$.  From (4.10) we have
\begin{multline}
F_{{\bf b}_i}^{({\bf b}_j)}(\lambda) = \\ 
\lambda^{s_0^{(j)}} \sum_{s=0}^\infty \frac{\displaystyle \prod_{k=1}^m \bigg(\frac{v_k^{(j)}-s_k^{(j)}}{\ell_k}\bigg)_s\bigg(\frac{v_k^{(j)}-s_k^{(j)}+1}{\ell_k}\bigg)_s\cdots\bigg(\frac{v_k^{(j)}-s_k^{(j)}+\ell_k-1}{\ell_k}\bigg)_s}{\displaystyle \bigg(\frac{s_0^{(j)}+1}{\ell_0}\bigg)_s\bigg(\frac{s_0^{(j)}+2}{\ell_0}\bigg)_s\cdots\bigg(\frac{s_0^{(j)}+\ell_0}{\ell_0}\bigg)_s}\lambda^{s\ell_0}.
\end{multline}

We first observe that there is a simple algebraic relation between the formula (6.7) and its special case $j=i$: since $s_k^{(i)}=0$ for $k=0,1,\dots,m$ we have
\begin{equation}
F_{{\bf b}_i}^{({\bf b}_i)}(\lambda) = 
\sum_{s=0}^\infty \frac{\displaystyle \prod_{k=1}^m \bigg(\frac{v_k^{(i)}}{\ell_k}\bigg)_s\bigg(\frac{v_k^{(i)}+1}{\ell_k}\bigg)_s\cdots\bigg(\frac{v_k^{(i)}+\ell_k-1}{\ell_k}\bigg)_s}{\displaystyle \bigg(\frac{1}{\ell_0}\bigg)_s\bigg(\frac{2}{\ell_0}\bigg)_s\cdots\bigg(\frac{\ell_0}{\ell_0}\bigg)_s}\lambda^{s\ell_0}.
\end{equation}

\begin{proposition}
The formula $(6.7)$ for $F_{{\bf b}_i}^{({\bf b}_j)}(\lambda)$ is obtained from the formula $(6.8)$ for $F_{{\bf b}_i}^{({\bf b}_i)}(\lambda)$ by multiplying the series $(6.8)$ by $\lambda^{s_0^{(j)}}$ and adding $s_0^{(j)}/\ell_0$ to each term inside  parenthesis.
\end{proposition}

\begin{proof}
The only nonobvious part of this proposition is the assertion that
\begin{equation}
\frac{v_k^{(i)}}{\ell_k} + \frac{s_0^{(j)}}{\ell_0} = \frac{v^{(j)}_k - s_k^{(j)}}{\ell_k}
\end{equation}
for $k=1,\dots,m$.  From (6.6) we have
\begin{equation}
 \sum_{k=1}^m (v_k^{(j)}-s_k^{(j)}){\bf a}_k = {\bf b}_i + s_0^{(j)}{\bf a}_0. 
 \end{equation}
 Using (1.1), this equation can be rewritten as
\begin{equation}
 \sum_{k=1}^m \bigg(v_k^{(j)}-s_k^{(j)}-\frac{s_0^{(j)}\ell_k}{\ell_0}\bigg){\bf a}_k = {\bf b}_i. 
\end{equation}
But ${\bf b}_i = \sum_{k=1}^m v_k^{(i)}{\bf a}_k$ and $\{{\bf a}_k\}_{k=1}^m$ is a linearly independent set, so
\[ v_k^{(i)} = v_k^{(j)}-s_k^{(j)}-\frac{s_0^{(j)}\ell_k}{\ell_0} \]
for $k=1,\dots,m$.  This equation is equivalent to (6.10).
\end{proof}

We now determine which terms in the ratio on the right-hand side of (6.8) cancel.
\begin{lemma}
{\bf (a)}  Suppose that either ${\bf b}_i$ is an interior point of $P(A)$ or ${\bf b}_i \neq {\bf b}_j$.  Then $0<v_k^{(j)}-s_k^{(j)}<\ell_k+1$ for $k=1,\dots,m$.  \\
{\bf (b)}  If ${\bf b}_i$ is a boundary point of $P(A)$ and ${\bf b}_i = {\bf b}_j$, then $0\leq v_k^{(j)}-s_k^{(j)}< 1$ for $k=1,\dots,m$.
\end{lemma}

\begin{proof}
If ${\bf b}_i$ is an interior point of $C(A)$, then, since $s_0^{(j)}\geq 0$, the right-hand side of (6.11) is an interior point of $C(A)$, so $v_k^{(j)}-s_k^{(j)}>0$ for all $k$.  If ${\bf b}_i\neq {\bf b}_j$, then $s_0^{(j)}>0$ so again the right-hand side of (6.11) is an interior point of $C(A)$.  Since ${\bf b}_i\in P(A)$, Equation (6.12) implies that
\[ v_k^{(j)}-s_k^{(j)}-\frac{s_0^{(j)}\ell_k}{\ell_0}<1 \]
for all $k$, i.~e., 
\[ \frac{v_k^{(j)}-s_k^{(j)}}{\ell_k} - \frac{s_0^{(j)}}{\ell_0}<\frac{1}{\ell_k}. \]
Since $s_0^{(j)}/\ell_0<1$, this implies that $v_k^{(j)}-s_k^{(j)}<\ell_k+1$.  This proves part (a).  

If ${\bf b}_i = {\bf b}_j$, then $s_k^{(j)}=0$ for $k=0,1,\dots,m$, so $v_k^{(j)} - s_k^{(j)} = v_k^{(j)}$.  Part (b) then follows from (6.5).
\end{proof}

\begin{proposition}
{\bf (a)}  Suppose that ${\bf b}_i$ is an interior point of $P(A)$.  There are exactly $\ell_0-R_k$ factors that appear in both numerator and denominator on the right-hand side of $(6.8)$, namely,
\[ \{(1-s^{(j)}_0/\ell_0)_s\mid \text{${\bf b}_j$ is a boundary point of $P(A)$}\}. \]
{\bf (b)}  Suppose that ${\bf b}_i$ is a boundary point of $P(A)$.  There are exactly $\ell_0-R_k-1$ factors that appear in both numerator and denominator on the right-hand side of~$(6.8)$, namely,
\[ \{(1-s^{(j)}_0/\ell_0)_s\mid \text{${\bf b}_j$ is a boundary point of $P(A)$ and ${\bf b}_j\neq {\bf b}_i$}\}. \]
\end{proposition}

\begin{proof}
Since $s^{(j)}_0$ runs through the set $\{0,1,\dots,\ell_0-1\}$ as ${\bf b}_j$ runs through the set~${\mathcal B}_k$, the ratios
$1-s^{(j)}_0/\ell_0$ run through the set $\{1/\ell_0,2/\ell_0,\dots,\ell_0/\ell_0\}$.  We need to determine when
\begin{equation}
 1-\frac{s^{(j)}_0}{\ell_0} = \frac{v^{(i)}_k + \sigma_k}{\ell_k} 
 \end{equation}
for some $k$ and some $\sigma_k\in\{0,1,\dots,\ell_k-1\}$.  Solving (6.10) for $v^{(i)}_k/\ell_k$, substituting into (6.15), and simplifying gives the equation
\begin{equation}
\ell_k - (v_k^{(j)}-s_k^{(j)}) = \sigma_k.
\end{equation}
Since $\ell_k$ and $s_k^{(j)}$ are integers, this equation has a solution $\sigma_k\in{\mathbb Z}$ if only if $v_k^{(j)}$ is an integer.  But $0\leq v_k^{(j)}<1$, so $v_k^{(j)}$ is integral if and only if $v_k^{(j)}=0$, i.~e., if and only if ${\bf b}_j$ is a boundary point of $P(A)$.  In this case we have $\sigma_k = \ell_k+s^{(j)}_k$, so we get cancellation if and only if $s^{(j)}_k\in\{-1,-2,\dots,-\ell_k\}$.  

If ${\bf b}_i$ is an interior point of $P(A)$, it follows from Lemma~6.13(a) that if $v_k^{(j)}=0$ then $s^{(j)}_k\in\{-1,-2,\dots,-\ell_k\}$.  Thus every boundary point ${\bf b}_j$ produces a cancellation.  This proves part (a) of the proposition.

If ${\bf b}_i$ is a boundary point of $P(A)$ and ${\bf b}_j$ is a boundary point different from~${\bf b}_i$, then Lemma~6.13(a) again implies that $s^{(j)}_k\in\{-1,-2,\dots,-\ell_k\}$ so we get a cancellation.  But if ${\bf b}_j={\bf b}_i$, then by Lemma~6.13(b) if $v^{(j)}_k=0$, then $s_k^{(j)} = 0$ so there is no cancellation in that case.  This proves part (b) of the proposiiton.
\end{proof}

Suppose that ${\bf b}_i$ is an interior point of $P(A)$. Applying Proposition 6.14(a) we let $\beta_1,\dots,\beta_{R_k}$ be the elements on the list $1/\ell_0,\dots,\ell_0/\ell_0$ that are not of the form $1-s_0^{(j)}/\ell_0$ for ${\bf b}_j$ a boundary point of $P(A)$.  In particular, since $\ell_0/\ell_0=1$ is not of this form we may take $\beta_1 = 1$.  We let 
$\alpha_1,\dots,\alpha_{R_k}$ be the elements remaining on the list
\[ \frac{v_1^{(i)}}{\ell_1},\frac{v_1^{(i)}+1}{\ell_1},\dots,\frac{v_1^{(i)}+\ell_1-1}{\ell_1},\ldots,\frac{v_m^{(i)}}{\ell_m},
\frac{v_m^{(i)}+1}{\ell_m},\dots,\frac{v_m^{(i)}+\ell_m-1}{\ell_m} \]
after one copy of each element $1-s_0^{(j)}/\ell_0$ with ${\bf b}_j$ a boundary point of $P(A)$ has been removed.
\begin{theorem}
Suppose that ${\bf b}_i$ is an interior point of $P(A)$.  
The $R_k$ series $(6.7)$ with ${\bf b}_j$ an interior point of $P(A)$ come from a full set of solutions at $x=0$ to the hypergeometric operator
\[ (\delta_x+\beta_1-1)\cdots(\delta_x + \beta_{R_k}-1) - x(\delta_x+\alpha_1)\cdots(\delta_x+\alpha_{R_k}) \]
by replacing $x$ by $\lambda^{\ell_0}$.
\end{theorem}

\begin{proof}
Propositions 6.9 and 6.14(a) imply that the series (6.7) are obtained from the series (6.8) in exactly the way that the series (1.6) are obtained from the series (1.5).  The conclusion of the theorem then follows from the discussion of the solutions of (1.4) in Section~1.
\end{proof}

Suppose that ${\bf b}_i$ is a boundary point of $P(A)$.  Applying Proposition 6.14(b) we let $\beta_1,\dots,\beta_{R_k+1}$ be the elements on the list $1/\ell_0,\dots,\ell_0/\ell_0$ that are not of the form $1-s_0^{(j)}/\ell_0$ for ${\bf b}_j$ a boundary point of $P(A)$ with ${\bf b}_j\neq{\bf b}_i$.  In particular, since $\ell_0/\ell_0=1$ is not of this form we may take $\beta_1 = 1$.  We let 
$\alpha_1,\dots,\alpha_{R_k+1}$ be the elements remaining on the list
\[ \frac{v_1^{(i)}}{\ell_1},\frac{v_1^{(i)}+1}{\ell_1},\dots,\frac{v_1^{(i)}+\ell_1-1}{\ell_1},\ldots,\frac{v_m^{(i)}}{\ell_m},
\frac{v_m^{(i)}+1}{\ell_m},\dots,\frac{v_m^{(i)}+\ell_m-1}{\ell_m} \]
after one copy of each element $1-s_0^{(j)}/\ell_0$ for ${\bf b}_j$ a boundary point of $P(A)$ with ${\bf b}_j\neq {\bf b}_i$ has been removed.  The  proof of the following result is analogous to the proof of Theorem 6.17.  
\begin{theorem}
Suppose that ${\bf b}_i$ is a boundary point of $P(A)$.  
The $R_k+1$ series $(6.7)$ with ${\bf b}_j$ an interior point of $P(A)$ or ${\bf b}_j = {\bf b}_i$ come from a full set of solutions at $x=0$  to the hypergeometric equation
\[ (\delta_x+\beta_1-1)\cdots(\delta_x + \beta_{R_k+1}-1) - x(\delta_x+\alpha_1)\cdots(\delta_x+\alpha_{R_k+1}) \]
by replacing $x$ by $\lambda^{\ell_0}$.
\end{theorem}

\section{Example}

We treat an example related to that of G\"{a}hrs \cite[Section 2.3]{G}.  Take ${\bf a}_1 = (5,1,0,0)$, ${\bf a}_2 = (0,4,1,0)$, ${\bf a}_3 = (0,0,8,0)$, ${\bf a}_4 = (0,0,0,2)$, and ${\bf a}_0 = (1,1,1,1)$.  Equation (1.1) becomes
\begin{equation}
10{\bf a}_0 = 2{\bf a}_1 + 2{\bf a}_2 + {\bf a}_3 + 5{\bf a}_4,
\end{equation}
so $\ell_1=\ell_2 = 2$, $\ell_3 = 1$, $\ell_4 = 5$, $\ell_0 = 10$, and
\[ f_\lambda = 2x_1^5x_2 + 2x_2^4x_3 + x_3^8 + 5x_4^2 -10\lambda x_1x_2x_3x_4. \]
We take ${\bf b}_1 = (1,1,1,1) = \frac{1}{5}{\bf a}_1 + \frac{1}{5}{\bf a}_2 + \frac{1}{10}{\bf a}_3 + \frac{1}{2}{\bf a}_4$, an interior point of $P(A)$, so
Equation~(6.8) becomes
\begin{multline}
F_{{\bf b}_1}^{{\bf b}_1}(\lambda)  = \\
  \sum_{s=0}^\infty \frac{\displaystyle \bigg(\frac{1}{10}\bigg)_s\bigg(\frac{6}{10}\bigg)_s\bigg(\frac{1}{10}\bigg)_s\bigg(\frac{6}{10}\bigg)_s\bigg(\frac{1}{10}\bigg)_s\bigg(\frac{1}{10}\bigg)_s\bigg(\frac{3}{10}\bigg)_s\bigg(\frac{5}{10}\bigg)_s\bigg(\frac{7}{10}\bigg)_s\bigg(\frac{9}{10}\bigg)_s}{\displaystyle\bigg(\frac{1}{10}\bigg)_s\bigg(\frac{2}{10}\bigg)_s\bigg(\frac{3}{10}\bigg)_s\bigg(\frac{4}{10}\bigg)_s\bigg(\frac{5}{10}\bigg)_s\bigg(\frac{6}{10}\bigg)_s\bigg(\frac{7}{10}\bigg)_s\bigg(\frac{8}{10}\bigg)_s\bigg(\frac{9}{10}\bigg)_s\bigg(\frac{10}{10}\bigg)_s}\lambda^{10s}\\
 = {}_4F_3(1/10,1/10,1/10,3/5;1/5,2/5,4/5;\lambda^{10}).
\end{multline}
By Theorem 6.17, the four series corresponding to the four interior points of $P(A)$ that are congruent to ${\bf b}_1$ modulo ${\mathbb Z}A_+$ and given by Equation (6.7) are obtained from a full set of solutions to
\begin{equation}
\textstyle \delta_x(\delta_x-\frac{1}{5})(\delta_x -\frac{3}{5})(\delta_x-\frac{4}{5}) - x (\delta_x + \frac{1}{10})^3(\delta_x + \frac{3}{5})
\end{equation}
by replacing $x$ by $\lambda^{10}$.  To find the equation satisfied by (7.2), make the change of variable $x\to \lambda^{10}$ in (7.3):
\begin{equation}
\delta_\lambda(\delta_\lambda - 2)(\delta_\lambda-6)(\delta_\lambda-8) -\lambda^{10} (\delta_\lambda + 1)^3(\delta_\lambda+6),
\end{equation}
where $\delta_\lambda = \lambda\frac{\partial}{\partial\lambda}$.  We show in the next section that (7.4) is the differential equation satisfied by $x^{{\bf b}_1}\in{\mathcal W}'$.  

For any two complex numbers $c_1$ and $c_2$ we have an equality of operators $(\delta_\lambda + c_1)\circ \lambda^{c_2} = \lambda^{c_2} \circ (\delta_\lambda + c_1 + c_2)$.  Taking $c_2=-1$, it follows that the series $\lambda\,{}_4F_3(\frac{1}{10},\frac{1}{10},\frac{1}{10}\frac{3}{5};\frac{1}{5},\frac{2}{5},\frac{4}{5};\lambda^{10})$ satisfies the equation
\begin{equation}
(\delta_\lambda-1)(\delta_\lambda - 3)(\delta_\lambda-7)(\delta_\lambda-9) -\lambda^{10} \delta_\lambda^3(\delta_\lambda+5).
\end{equation}

Equation (7.5) differs from the equation computed by G\"{a}hrs (\cite[unnumbered equation following Equation~(2.4)]{G}) by some constant factors, due to the fact that our normalization of $f_\lambda$ differs from that of G\"{a}hrs.  The shift from our equation (7.4) to the G\"{a}hrs equation (7.5) is owing to the fact that G\"{a}hrs is computing, in the notation of \cite{G}, the differential equation satisfied by $s\Omega_0/f$ rather than the equation satisfied by $\Omega_0/f$.  

To further illustrate our results, we compute the elements of ${\mathcal B}$ that are congruent to ${\bf b}_1$ modulo ${\mathbb Z}A_+$:
\begin{align*}
{\bf b}_2:=(2,2,2,0)&\equiv {\bf b}_1 + {\bf a}_0 & {\bf b}_3:= (3,3,3,1)&\equiv {\bf b}_1+2{\bf a}_0 & {\bf b}_4:=(4,4,4,0)&\equiv {\bf b}_1 + 3{\bf a}_0 \\
{\bf b}_5:=(0,0,4,1)&\equiv{\bf b}_1 + 4{\bf a}_0 & {\bf b}_6:=(1,1,5,0)&\equiv{\bf b}_1 +5{\bf a}_0 & {\bf b}_7:=(2,2,6,1)&\equiv{\bf b}_1 + 6{\bf a}_0 \\
{\bf b}_8:=(3,3,7,0)&\equiv{\bf b}_1 + 7{\bf a}_0 & {\bf b}_9:=(4,4,8,1)&\equiv {\bf b}_1 + 8{\bf a}_0 & {\bf b}_{10}:=(0,0,0,0)&\equiv{\bf b}_1 + 9{\bf a}_0.
\end{align*}
Of these, ${\bf b}_j$ is an interior point of $P(A)$ for $j=1,3,7,9$ (so $R=4$) and is a boundary point for $j=2,4,5,6,8,10$.  By Proposition~6.14(a), the factors in (7.2) that cancel are 
\[ \bigg\{(1-\frac{s_0^{(j)}}{\ell_0})\mid j=2,4,5,6,8,10\bigg\} = \bigg\{\frac{1}{10},\frac{3}{10},\frac{5}{10},\frac{6}{10},\frac{7}{10},\frac{9}{10}\bigg\}. \]

\section{Differential equation satisfied by $x^{\bf b}$}

Let $u\in M_k$ (for some $k$) and let $R_k$ be the number of elements of ${\mathcal B}_k$ that are interior to $P(A)$ and congruent to $u$ modulo ${\mathbb Z}A_+$.  In the next section, we prove the following result.
\begin{proposition}
{\bf (a)}  If $u\in M_k$ is an interior point of $C(A)$, then $x^u\in{\mathcal W}'_k$ satisfies a differential equation of order $R_k$. \\
{\bf (b)}  If $u\in M_k$ is a boundary point of $C(A)$, then $x^u\in{\mathcal W}'_k$ satisfies a differential equation of order $R_k+1$.
\end{proposition}

Let ${\bf b}\in{\mathcal B}_k$ for some $k$ and suppose that ${\bf b}$ is an interior point of $P(A)$.  By Proposition~8.1(a), there is a differential operator
\[ \partial_\lambda^{R_k} + c_{R_k-1}(\lambda)\partial_\lambda^{R_k-1} + \cdots + c_0(\lambda), \]
where the $c_i(\lambda)$ lie in ${\mathbb C}(\lambda)$, that annihilates $x^{\bf b}\in{\mathcal W}'$.  (Recall that $\partial_\lambda$ acts on ${\mathcal W}'_k$ as $D_\lambda$.)  If $\xi\in{\rm Hom}_{\mathcal D}({\mathcal W}_k,{\mathbb C}[[\lambda]])$, then this differential operator will also annihilate the series $\xi(x^{\bf b})$.  By Section~6, as ${{\bf b}'}$ runs through the elements of ${\mathcal B}_k$, the power series $\xi_{{\bf b}'}(x^{\bf b})$ that are nonzero (which occurs exactly when ${\bf b}'$ is an interior point of $P(A)$) run through the nonzero elements of row ${\bf b}$ of the fundamental solution matrix computed in Section~6.  A similar argument applies when ${\bf b}$ is a boundary point of $P(A)$ by using Proposition 8.1(b).  This gives the following result.

\begin{theorem}
For ${\bf b}\in{\mathcal B}_k$, the element $x^{\bf b}\in{\mathcal W}'_k$ satisfies the differential equation whose complete set of solutions is given by the nonzero elements of row ${\bf b}$ of the fundamental solution matrix computed in Section~$6$.
\end{theorem}

\section{Proof of Proposition 8.1}

Proposition 8.1 is a consequence of the following result.  Let $\widehat{\mathcal W}_k\subset{\mathcal W}_k'$ be the ${\mathbb C}(\lambda)$-subspace generated by those $x^u$, $u\in M_k$, for which $u$ is an interior point of~$C(A)$.
\begin{proposition}
$\dim_{{\mathbb C}(\lambda)} \widehat{\mathcal W}_k = R_k$.
\end{proposition}

\begin{proof}[Proof of Proposition $8.1$]
If $u$ is an interior point of $C(A)$, then the $R_k+1$ elements
\[ x^u,D_\lambda(x^u)=-\ell_0x^{{\bf a}_0},\dots,D_\lambda^{R_k}(x^u)= (-\ell_0)^{R_k}x^{u+R_k{\bf a}_0} \]
 are interior points of $C(A)$, so by Proposition 9.1 there is a ${\mathbb C}(\lambda)$-linear relation between them.  This proves part~(a).  

If $u$ is a boundary point of $C(A)$, then $D_\lambda(x^u), D_\lambda^2(x^u), \dots,D_\lambda^{R_k+1}(x^u)$ are interior points of $C(A)$, so by Proposition 9.1 there is a ${\mathbb C}(\lambda)$-linear relation between them.  This proves part~(b).  
\end{proof}

Proposition 9.1 is a consequence of the appendix to \cite{AS0}, which contains a more general result \cite[Theorem A.9]{AS0}.  We give here a simpler version of that result.  

For any subset $B\subseteq A$, let $V_B\subseteq V$ be the subspace generated by $B$, so $\dim_{{\mathbb R}} V_B={\rm card}(B)$.  Let $V_{B,{\mathbb Z}} = V_B\cap{\mathbb Z}$ and let $C(B)$ be the real cone generated by $B$.  The cone $C(B)$ is a face of $C(A)$ of dimension ${\rm card}(B)$.  Put $M_B = V_{B,{\mathbb Z}}\cap C(B)$ and let $S'_B$ be the ${\mathbb C}(\lambda)$-algebra generated by $\{x^u\mid u\in M_B\}$.  Let $\theta_B:S'\to S'_B$ be the ${\mathbb C}(\lambda)$-homomorphism defined by
\[ \theta_B(x^u) = \begin{cases} x^u & \text{if $u\in M_B$,} \\ 0 & \text{if $u\not\in M_B$.}\end{cases} \]
For $j=1,\dots,m$, we define $\tilde{D}^B_j = \theta_B\circ\tilde{D}_j$.  It follows from Lemma 5.29 that for $u\in M_B$, say $u=\sum_{j:{\bf a}_j\in B} u'_j{\bf a}_j$, we have
\begin{equation}
 \tilde{D}_j^B(x^u) = \begin{cases} u'_jx^u + \ell_jx^{u+{\bf a}_j} & \text{if ${\bf a}_j\in B$,} \\ 0 & \text{if ${\bf a}_j\not\in B$.} \end{cases} 
 \end{equation}

Consider ${\rm Kos}^\bullet(S'_B,\{\tilde{D}^B_j\}_{{\bf a}_j\in B})$, the Koszul complex on $S'_B$ defined by the operators~$\{\tilde{D}_j\}_{{\bf a}_j\in B}$.  It follows from (9.2) that the associated graded complex is the Koszul complex on $S_B'$ defined by the $\{\ell_j{\bf a}_j\}_{{\bf a}_j\in B}$.  Since these monomials form a regular sequence on $S'_B$, both these Koszul complexes are acyclic except in top dimension, where we have
\begin{equation}
\dim_{{\mathbb C}(\lambda)}H^{|B|}({\rm Kos}^\bullet(S'_B,\{\tilde{D}^B_j\}_{{\bf a}_j\in B}))= |{\mathcal B}\cap C(B)|.
\end{equation}

We define subcomplexes of these Koszul complexes.  
Let $B_1$ be a subset of $B$ and let $S_{B,B_1}'$ be the $C(\lambda)$-subalgebra of $S_B'$ generated by those $x^u$, $u\in M_B$, for which $u=\sum_{j:{\bf a}_j\in B} u'_j{\bf a}_j$ with $u'_j>0$ when ${\bf a}_j\in B_1$.  

It follows from (9.2) that $\tilde{D}_j$ for ${\bf a}_j\in B$ maps $S_{B,B_1}$ to $S_{B,B_1\cup\{{\bf a}_j\}}$.  This allows us to define a subcomplex of ${\rm Kos}^\bullet(S'_B,\{\tilde{D}^B_j\}_{{\bf a}_j\in B})$.  Define
\[ {\mathcal K}^k(B,B_1) = \bigoplus_{\substack{{\bf a}_{j_1},\dots,{\bf a}_{j_k}\in B\\ j_1<\dots<j_k}}  S'_{B,B_1\cap \{{\bf a}_{j_1}\dots,{\bf a}_{j_k}\}} e_{j_1}\wedge\dots\wedge e_{j_k}. \]
Then ${\mathcal K}^\bullet(B,B_1)$ is a subcomplex of ${\rm Kos}^\bullet(S'_B,\{\tilde{D}^B_j\}_{{\bf a}_j\in B})$.  Note that, in particular, 
\begin{equation}
{\mathcal K}^\bullet(B,\emptyset) = {\rm Kos}^\bullet(S'_B,\{\tilde{D}^B_j\}_{{\bf a}_j\in B}).
\end{equation}

Let ${\bf a}_j\in B_1$.  The map $\theta_{B\setminus\{{\bf a}_j\}}$ maps $S'_B$ onto $S'_{B\setminus\{{\bf a}_j\}}$ and induces a surjective map 
\[ {\mathcal K}^\bullet(B,B_1\setminus\{{\bf a}_j\})\to {\mathcal K}^\bullet(B\setminus\{{\bf a}_j\},B_1\setminus\{{\bf a}_j\})[1] \]
 whose kernel is ${\mathcal K}^\bullet(B,B_1)$.  There is thus a short exact sequence
\begin{equation}
0\to {\mathcal K}^\bullet(B,B_1) \to {\mathcal K}^\bullet(B,B_1\setminus\{{\bf a}_j\})\to {\mathcal K}^\bullet(B\setminus\{{\bf a}_j\},B_1\setminus\{{\bf a}_j\})[1]\to 0.
\end{equation}
From (9.3) and (9.4) we have
\begin{equation}
 H^i({\mathcal K}^\bullet(B,\emptyset)) = 0\quad\text{for $i=0,1,\dots,|B|-1$} 
 \end{equation}
and
\begin{equation}
\dim_{{\mathbb C}(\lambda)}H^{|B|}({\mathcal K}^\bullet(B,\emptyset))= |{\mathcal B}\cap C(B)|.
\end{equation}
Using the long exact sequence of cohomology associated to (9.5) and induction on $|B_1|$ (the case $|B_1|=0$ being (9.6) and (9.7)) gives the following result.
\begin{theorem}
We have 
\begin{equation}
H^i({\mathcal K}^\bullet(A,A)) = 0\quad\text{for $i=0,1,\dots,m-1$}
\end{equation}
and
\begin{equation}
 \dim_{{\mathbb C}(\lambda)}  H^m({\mathcal K}^\bullet(A,A)) = \sum_{B\subseteq A}(-1)^{|A\setminus B|} |{\mathcal B}\cap C(B)|. 
 \end{equation}
\end{theorem}

Note that the right-hand side equals $|{\mathcal B}\cap C(A)^\circ|$, the number of points of ${\mathcal B}$ that lie in the interior of the cone $C(A)$.
\begin{corollary}
We have
\[ \dim_{{\mathbb C}(\lambda)}  H^m({\mathcal K}^\bullet(A,A)) = |{\mathcal B}\cap C(A)^\circ|. \]
\end{corollary}

\begin{proof}[Proof of Proposition $9.1$]
Since the elements of ${\mathcal B}\cap C(A)^\circ$ are linearly independent in $\widehat{\mathcal W}$, we have
\[ \dim_{{\mathbb C}(\lambda)} \widehat{\mathcal W}\geq |{\mathcal B}\cap C(A)^\circ|. \]
On the other hand, there is a natural surjection
\[ H^m({\mathcal K}^\bullet(A,A)) = S_{A,A}'/\sum_{j=1}^m \tilde{D}_jS'_{A,A\setminus\{{\bf a}_j\}}\to \widehat{W} = S'_{A,A}/\big(S'_{A,A}\cap \sum_{j=1}^m \tilde{D}_j S_{A,\emptyset}\big), \]
which implies that 
\[ \dim_{{\mathbb C}(\lambda)} \widehat{\mathcal W}\leq |{\mathcal B}\cap C(A)^\circ|. \]
Proposition 9.1 now follows from the direct sum decomposition discussed in Section~4.
\end{proof}

\section{Solutions at infinity}

We begin with some general remarks.  Consider a differential operator
\begin{equation}
(\delta_x+\beta_1-1)\cdots(\delta_x+\beta_n-1) - x(\delta_x + \alpha_1)\cdots(\delta_x + \alpha_n).
\end{equation}
Let $F(x)$ be a solution of this operator and let $\ell$ be a positive integer.
\begin{lemma}
The function $F(\lambda^\ell)$ satisfies the differential operator 
\begin{equation}
(\delta_\lambda +\ell\beta_1 -\ell)\cdots(\delta_\lambda + \ell\beta_n-\ell) - \lambda^\ell(\delta_\lambda + \ell\alpha_1)\cdots (\delta_\lambda+\ell\alpha_n).
\end{equation}
\end{lemma}

\begin{proof}
Make the change of variable $x = \lambda^\ell$ in (10.1).
\end{proof}

We want to describe the solutions of (10.3) at infinity.  Consider the differential operator
\begin{equation}
(\delta_x +\tau_1-1)\cdots(\delta_x + \tau_n-1) -x(\delta_x +\sigma_1)\cdots(\delta_x+\sigma_n).
\end{equation}
Let $G(x)$ be a solution of this differential operator.  
\begin{lemma}
The function $G(\lambda^{-\ell})$ satisfies the differential operator
\begin{equation}
(\delta_\lambda - \ell\sigma_1)\cdots(\delta_\lambda - \ell\sigma_n) -\lambda^\ell(\delta_\lambda - \ell\tau_1+\ell)\cdots(\delta_\lambda - \ell\tau_n+\ell).
\end{equation}
\end{lemma}

\begin{proof}
Make the change of variable $x=\lambda^{-\ell}$ in (10.4).  
\end{proof}

\begin{corollary}
Let $G(x)$ be a solution of the differential operator
\begin{equation}
(\delta_x -\alpha_1)\cdots(\delta_x -\alpha_n) -x(\delta_x +1-\beta_1)\cdots(\delta_x+1-\beta_n).
\end{equation}
Then $G(\lambda^{-\ell})$ satisfies $(10.3)$.
\end{corollary}

\begin{proof}
Taking $\sigma_i = 1-\beta_i$ and $\tau_i = 1-\alpha_i$ in (10.4) transforms (10.6) into (10.3).
\end{proof}

To make the solutions of (10.8) slightly simpler to express, we make one further change.  
Assume that the $\alpha_i$ are real numbers and that, say, $\alpha_1\leq \alpha_i$ for all $i$.  Consider the differential operator obtained from (10.8) by adding $\alpha_1$ to each term:
\begin{equation}
\delta_x(\delta_x-\alpha_2+\alpha_1)\cdots(\delta_x -\alpha_n+\alpha_1) -x(\delta_x +1-\beta_1+\alpha_1)\cdots(\delta_x+1-\beta_n+\alpha_1).
\end{equation}
One checks that (10.9) is the operator obtained from (10.8) by multiplying (10.8) on the right by $x^{\alpha_1}$ and on the left by $x^{-\alpha_1}$.  This implies that if $H(x)$ is a solution of the operator (10.9), then $x^{\alpha_1}H(x)$ is a solution of the operator (10.8).  Corollary~10.7 then implies the following result.
\begin{corollary}
Let $H(x)$ be a solution of the operator $(10.9)$.  Then $\lambda^{-\ell\alpha_1}H(\lambda^{-\ell})$ satisfies $(10.3)$.
\end{corollary}

{\bf Example.}  We illustrate these observations with the example from Section 7.  If we take $n=4$ and $\beta_1 = 1$, $\beta_2 = 4/5$, $\beta_3 = 2/5$, $\beta_4=1/5$, $\alpha_1 = \alpha_2=\alpha_3 = 1/10$, and $\alpha_4 = 3/5$ in (10.1), then (10.1) becomes (7.3).  If we then take $\ell=10$ in (10.3), then (10.3) becomes (7.4).  Equation (10.9) then becomes
\begin{equation}
\textstyle \delta_x^3(\delta_x-\frac{1}{2}) - x(\delta_x + \frac{1}{10})(\delta_x + \frac{3}{10}) (\delta_x + \frac{7}{10})(\delta_x + \frac{9}{10}).
\end{equation}
Corollary 10.10 implies that we can get a full set of solutions to (7.4) at infinity by taking a full set of solutions to (10.11) at the origin, replacing the variable by $\lambda^{-10}$, and multiplying by $\lambda^{-1}$.  For example, one solution to (10.11) is
${}_4F_3(\frac{1}{10},\frac{3}{10},\frac{7}{10},\frac{9}{10};\frac{1}{2},1,1;x)$, so one solution to (7.4) at infinity is
\[ \textstyle \lambda^{-1}{}_4F_3(\frac{1}{10},\frac{3}{10},\frac{7}{10},\frac{9}{10};\frac{1}{2},1,1;\lambda^{-10}). \]

\section{Connection with de Rham cohomology}

We describe here an example where we apply the ideas of this paper to compute a fundamental solution matrix for the Picard-Fuchs equation of a projective variety.  Consider the projective hypersurface $X\subseteq{\mathbb P}^{n-1}_{{\mathbb C}(\lambda)}$ defined by the equation
\[ f_\lambda = \sum_{j=1}^n w_jx_j^d - d\lambda x_1^{w_1}\cdots x_n^{w_n} = 0, \]
where $w_j\geq 1$ for $j=1,\dots,n$ and $\sum_{j=1}^n w_j = d$ (see Katz \cite{K2} for an $\ell$-adic treatment of this hypersurface).  

Let $\widetilde{\mathcal W}\subseteq{\mathcal W}'$ be the subspace generated by those $x^u$, $u=(u_1,\dots,u_n)\in M$, for which
$u_j\geq 1$ for $j=1,\dots,n$ and $\sum_{j=1}^n u_j$ is a multiple of $d$.  Note that $M$ in this case is the set of all lattice points in the first orthant.  By Katz \cite[Corollary~1.15]{K}, there is an isomorphism of ${\mathcal D}'$-modules $\widetilde{W}\cong H^{n-1}_{\rm dR}({\mathbb P}^{n-1}_{{\mathbb C}(\lambda)}\setminus X)$.  Specifically, the element $x^u\in\widetilde{\mathcal W}$ with $\sum_{j=1}^n u_j = kd$ corresponds (up to a scalar multiple) to the $(n-1)$-form
\[ \frac{x^u}{f_\lambda^k}\sum_{j=1}^n (-1)^j \frac{dx_1}{x_1}\cdots\widehat{\frac{dx_j}{x_j}}\cdots\frac{dx_n}{x_n}. \]

One can thus obtain a fundamental solution matrix for the Picard-Fuchs equation of $H^{n-1}_{\rm dR}({\mathbb P}^{n-1}_{{\mathbb C}(\lambda)}\setminus X)$, or, via the residue map, of the primitive part of $H^{n-2}_{\rm dR}(X)$, using the results of this paper.  By Section~9, a basis for $\widetilde{W}$ is 
\[ \widetilde{\mathcal B} = \bigg\{ x^u\mid \text{$u=(u_1,\dots,u_n)$ with $0<u_j<d$ for $j=1,\dots,n$ and $d$ divides $\sum_{j=1}^n u_j$}\bigg\}. \]
The set $\widetilde{\mathcal B}$ decomposes following the decomposition of ${\mathcal W}'$ in Section~4.  One can then write down the solution matrix of the Picard-Fuchs equation using Equations~(6.7) and~(4.9).

{\bf Example.}  Let $X\subseteq{\mathbb P}^2_{{\mathbb C}(\lambda)}$ be the plane curve of genus 10 defined by the equation
\[ f_\lambda = x_1^6+2x_2^6+3x_3^6-6\lambda x_1x_2^2x_3^3=0. \]
Thus ${\bf a}_1 = (6,0,0)$, ${\bf a}_2=(0,6,0)$, ${\bf a}_3=(0,0,6)$, and ${\bf a}_0 = (1,2,3)$.  The relation
\[ 6{\bf a}_0 = {\bf a}_1 + 2{\bf a}_2 + 3{\bf a}_3 \]
gives $\ell_1=1$, $\ell_2 = 2$, $\ell_3=3$, and $\ell_0 = 6$.  The set $\widetilde{\mathcal B}$ has 20 elements.  Under the decomposition of Section 4, it decomposes into 6 subsets:
\[ \widetilde{\mathcal B}_1 = \{ (1,2,3), (5,4,3)\},\quad \widetilde{\mathcal B}_2 = \{(2,1,3),(4,5,3)\}, \]
\[ \widetilde{\mathcal B}_3 = \{(2,2,2), (3,4,5),(5,2,5)\}, \quad \widetilde{\mathcal B}_4 = \{(3,2,1), (4,4,4), (1,4,1)\}, \]
\[ \widetilde{\mathcal B}_5 = \{(1,3,2),(2,5,5),(3,1,2),(4,3,5),(5,5,2)\},\]
\[ \widetilde{\mathcal B}_6 = \{(2,3,1),(3,5,4),(4,1,1),(5,3,4),(1,1,4)\}. \]
The Picard-Fuchs equation thus has a block diagonal form, with two $(2\times 2)$-blocks, two $(3\times 3)$-blocks, and two $(5\times 5)$-blocks.  We compute the solution matrix for one of these blocks, the block associated to $\widetilde{\mathcal B}_1$.  From Equations (6.7) and (4.9) we get
\[ \begin{pmatrix} {}_2F_1(\frac{1}{6},\frac{1}{6};\frac{1}{3};\lambda^6) & -27\lambda^4 {}_2F_1(\frac{5}{6},\frac{5}{6};\frac{5}{3};\lambda^6) \\
\vspace*{.01in}\\
-\frac{1}{2}\lambda^2{}_2F_1(\frac{7}{6},\frac{7}{6};\frac{4}{3};\lambda^6) & {}_2F_1(\frac{5}{6},\frac{5}{6};\frac{2}{3};\lambda^6) \end{pmatrix}. \]
The series in the first row are solutions of the operator
\[ \textstyle\delta_x(\delta_x-\frac{2}{3}) - x(\delta_x+\frac{1}{6})^2 \]
with the variable $x$ replaced by $\lambda^6$.  The series in the second row are solutions of the operator
\[ \textstyle\delta_x(\delta_x-\frac{1}{3}) - x(\delta_x+\frac{5}{6})^2 \]
with the variable $x$ replaced by $\lambda^6$.

\end{document}